\documentclass[10pt,reqno]{amsart}
\usepackage{amsfonts}
\usepackage{mathrsfs}  
\usepackage{amsthm}
\usepackage{amsxtra}
\usepackage{amssymb}
\usepackage{tikz}
\usepackage{enumitem}
\usepackage{caption}
\usepackage{graphicx}
\usepackage{subfigure}
\usepackage{tabularx}
\usepackage{multirow,booktabs}
\usepackage{float} 
\newtheorem{theorem}{Theorem}[section]

\newtheorem{lemma}[theorem]{Lemma}
\newtheorem{proposition}[theorem]{Proposition}
\newtheorem{assumption}{Assumption}

\newtheorem{example}[theorem]{Example}
\newtheorem{algorithm}{Algorithm}

\newcommand{\R}{\mathbb{R}}

\graphicspath{{image/}}

\def\EE{\mathbb E}
\def\P{\mathbb P}
\def\d{\delta}
\def\l{\langle}
\def\r{\rangle}

\begin{document}
	
\title[An adaptive heavy ball method for ill-posed inverse problems]
	{An adaptive heavy ball method for ill-posed inverse problems}
	

\author{Qinian Jin}
\address{Mathematical Sciences Institute, Australian National
University, Canberra, ACT 2601, Australia}
\email{qinian.jin@anu.edu.au} \curraddr{}
	
\author{Qin Huang}
\address{School of Mathematics and Statistics, Beijing Institute of Technology, 
Beijing, 100081, China}
\email{huangqin@bit.edu.cn}

	
	
	
\begin{abstract}
In this paper we consider ill-posed inverse problems, both linear and nonlinear, by a heavy ball 
method in which a strongly convex regularization function is incorporated to detect the feature of 
the sought solution. We develop ideas on how to adaptively choose the step-sizes and the momentum 
coefficients to achieve acceleration over the Landweber-type method. We then analyze the method and 
establish its regularization property when it is terminated by the discrepancy principle. Various 
numerical results are reported which demonstrate the superior performance of our method over the 
Landweber-type method by reducing substantially the required number of iterations and the computational time.
\end{abstract}
	
\def\d{\delta}
\def\P{\mathbb{P}}
\def\l{\langle}
\def\r{\rangle}
\def\la{\lambda}
\def\EE{{\mathbb E}}
\def\R{{\mathcal R}}
\def\a{\alpha}
\def\p{\partial}
\def\ep{\varepsilon}
\def\PP{{\mathbb P}}

\keywords{Ill-posed inverse problems, adaptive heavy-ball method, step-size, momentum coefficients, convergence}
	%
	%
	

	\maketitle
	
\section{\bf Introduction}

Consider an operator equation of the form 
\begin{align}\label{AHB.1}
F(x) = y,
\end{align}
where $F : \mbox{dom}(F)\subset X \to Y$ is an operator between two real Hilbert spaces $X$ and $Y$ 
with domain $\mbox{dom}(F)$. The operator $F$ could be either linear or nonlinear. Throughout, we 
always assume that (\ref{AHB.1}) has a solution, i.e. $y \in \mbox{Ran}(F)$, the range of $F$. 
We are interested in the situation that (\ref{AHB.1}) is ill-posed in the sense that 
its solution does not depend continuously on the right hand data. Such situation can happen in a 
broad range of applications of inverse problems. In practical scenarios, data are acquired 
through experiments and thus the exact data may not be available; instead we only have measurement 
data corrupted by noise. Due to the ill-posed nature of the underlying problem, it is therefore 
important to develop algorithms to produce stable approximate solutions of (\ref{AHB.1}) using 
noisy data.

In practical applications, some {\it a priori} information, such as non-negativity, sparsity and 
piecewise constancy, on the sought solution, is often available. Utilizing such feature 
information in algorithm design can significantly improve the reconstruction accuracy. By 
incorporating the available {\it a priori} information, one may pick a proper, lower semi-continuous, 
strongly convex function $\R: X \to (-\infty, \infty]$ as a regularization function to determine 
a solution of (\ref{AHB.1}) with the desired feature. Let $y^\d$ be a noisy data satisfying 
$$
\|y^\d - y\| \le \d
$$
with a known noise level $\d>0$. The Landweber-type method 
\begin{align}\label{Land}
\begin{split}
\xi_{n+1}^\d & = \xi_n^\d - \a_n^\d L(x_n^\d)^* (F(x_n^\d) - y^\d), \\
x_{n+1}^\d & = \arg\min_{x\in X} \left\{\R(x) - \l \xi_{n+1}^\d, x\r \right\},
\end{split}
\end{align}
has been considered in \cite{JW2013} for solving (\ref{AHB.1}), where $\a_n^\d>0$ is the 
step-size and $\{L(x): x \in \mbox{dom}(F)\}$ is a family of properly chosen bounded 
linear operators. In case $F$ is Fr\'{e}chect differentiable at $x$, we may take $L(x) = F'(x)$, 
the Fr\'{e}chet derivative of $F$ at $x$; otherwise, $L(x)$ should be appropriately chosen as a 
substitute for the non-existent Fr\'{e}chet derivative. This method has been thoroughly analyzed in \cite{Jin2016,JW2013} under
the constant step-size 
\begin{align}\label{css}
\a_n^\d = \frac{\mu_0}{L^2},
\end{align}
when an upper bound $L$ of $\|L(x)\|$ around the sought solution is available, and the adaptive 
step-size
\begin{align}\label{ass}
\a_n^\d = \min\left\{\frac{\mu_0\|F(x_n^\d) - y^\d\|^2}{\|L(x_n^\d)^*(F(x_n^\d)-y^\d)\|^2}, \mu_1\right\}
\end{align}
with suitable choices of the parameters $\mu_0$ and $\mu_1$, and the regularization property 
has been established when the iteration is terminated by the discrepancy principle. Note that, 
when $\R(x) = \|x\|^2/2$, the corresponding method with $\a_n^\d$ chosen by (\ref{css}) is the 
classical Landweber method which has been analyzed in \cite{EHN1996,HNS1995}, and the corresponding 
method with $\a_n^\d$ chosen by (\ref{ass}) is the minimal error method (\cite{KNS2008}). 
Extensive numerical simulations have demonstrated that the Landweber-type method (\ref{Land}) 
is a slowly convergent approach, often necessitating a large number of iteration steps before 
termination by the discrepancy principle. It is therefore important to develop strategies 
for accelerating the Landweber-type method (\ref{Land}) while maintaining its simple implementation 
feature. 

For linear ill-posed problems in Hilbert spaces, Brakhage used in \cite{B1987} the Jacobi 
polynomials to construct a method, which is known as the $\nu$-method, to accelerate the 
classical Landweber iteration. Additionally, in \cite{H1991}, a family of accelerated Landweber 
iterations was developed by employing orthogonal polynomials and the spectral theory of self-adjoint 
operators. However, the acceleration strategy using orthogonal polynomials is scarcely  
applicable for the Landweber-type method (\ref{Land}) when the forward operator $F$ is 
nonlinear or the regularization function $\R$ is non-quadratic. In \cite{HJW2015,SS2009} the 
sequential subspace optimization strategies were employed to accelerate the Landweber iteration. 
The implementation of this strategy however requires to solve an additional optimization 
problem to determine the weighted coefficients of the search directions at each iteration 
step which unfortunately has the drawback of slowing down the method. 

On the other hand, to determine a solution of well-posed minimization problem 
\begin{align}\label{MP}
\min_{x\in X} f(x)
\end{align}
with a continuous differentiable objective function $f$, tremendous progress has been made 
toward accelerating the gradient method 
$$
x_{n+1} = x_n - \a_n \nabla f(x_n)
$$
in optimization community. In particular, Nesterov's accelerated gradient method (\cite{N1983}) and 
Polyak's heavy ball method (\cite{P1964}) stood out and had far-reaching impact on the development 
of fast first order optimization methods. In Nesterov's accelerated gradient method, to define the 
next iterate, instead of using the current iterate, it uses a carefully chosen extrapolation point, 
built from the previous two iterates (\cite{AP2016,BT2009,CD2015,N2018}). Based on Nesterov's acceleration 
strategy, an accelerated version of Landweber-type method has been proposed in \cite{Jin2016} and a refined 
version of the method takes the form (see also \cite{Jin2022,ZWJ2018})
\begin{align}\label{NAG}
\begin{split}
\hat \xi_n^\d & = \xi_n^\d + \frac{n-1}{n+\gamma}(\xi_n^\d - \xi_{n-1}^\d), \qquad\quad\quad \, \,  
\hat x_n^\d  = \arg\min_{x\in X} \left\{ \R(x) - \l \hat \xi_n^\d, x\r\right\}, \\
\xi_{n+1}^\d & = \hat \xi_n^\d - \a_n^\d L(\hat x_n^\d)^* (F(\hat x_n^\d) - y^\d), \quad 
x_{n+1}^\d  = \arg\min_{x\in X} \left\{\R(x) - \l \xi_{n+1}^\d, x\r\right\},
\end{split}
\end{align}
where $\gamma \ge 2$. The numerical results presented in \cite{Jin2016} demonstrate its striking 
acceleration effect. Inspired by the numerical observations in \cite{Jin2016}, some efforts have 
been devoted to analyzing the regularization property of (\ref{NAG}) in the context of ill-posed 
problems. When X is a Hilbert space, $F$ is a bounded linear operator and $\R(x) = \|x\|^2/2$, 
the regularization property of the corresponding method, terminated by either {\it a priori} 
stopping rules or the discrepancy principle, has been established in \cite{K2021,N2017} based on 
a general acceleration framework in \cite{H1991} using orthogonal polynomials. When $F$ is nonlinear 
or $\R$ is non-quadratic, although some partial results are available in \cite{HR2018,Jin2022}, 
the analysis on (\ref{NAG}) is under developed and the understanding of its regularization property 
is largely open. By modifying the first equation in (\ref{NAG}) by $\hat \xi_n^\d  = \xi_n^\d 
+ \la_n^\d(\xi_n^\d - \xi_{n-1}^\d)$ with a connection coefficient $\la_n^\d$ to be determined, 
a two-point gradient method has been considered in \cite{HR2017,ZWJ2018}. This method, however, requires 
a discrete backtracking search algorithm to determine $\la_n^\d$ at each iteration step which can 
incur additional computational time. 

Different from Nesterov's method, Polyak's heavy ball method accelerates the gradient method by 
adding a momentum term to define the iterates, i.e. 
\begin{align}\label{HB0}
x_{n+1} = x_n - \a_n  \nabla f(x_n) + \beta_n (x_n - x_{n-1}). 
\end{align}
The performance of (\ref{HB0}) depends crucially on the property of $f$ and the choices of 
the step-size $\a_n$ and the momentum coefficient $\beta_n$, and its analysis turns out to 
be very challenging. When $f$ is twice continuous differentiable, strongly convex and has 
Lipschitz continuous gradient, it has been demonstrated in \cite{P1964}, under suitable 
constant choices of $\a_n$ and $\beta_n$, that the iterative sequence $\{x_n\}$ enjoys a provable 
linear convergence faster than the gradient descent. Due to the recent development of machine 
learning and the appearance of large scale problems, the heavy ball method has gained renewed 
interest and much attention has been paid on understanding its convergence behavior (\cite{GFJ2015,N2018,O2018,OCBP2014,ZK1993}). Although many strategies have been proposed 
to select $\a_n$ and $\beta_n$, how to choose these parameters to achieve acceleration effect 
for general convex and non-convex problems remains an active research topic in optimization 
community. 

In this paper we will consider the iterative method 
\begin{align}\label{AHB.2}
\begin{split}
\xi_{n+1}^\d & = \xi_n^\d - \a_n^\d L(x_n^\d)^* (F(x_n^\d) - y^\d) + \beta_n^\d (\xi_n^\d - \xi_{n-1}^\d), \\
x_{n+1}^\d & = \arg\min_{x\in X} \left\{\R(x) - \l \xi_{n+1}^\d, x\r\right\}
\end{split}
\end{align}
for solving ill-posed inverse problem (\ref{AHB.1}), where $\a_n^\d$ denotes the step-size and 
$\beta_n^\d$ denotes the momentum coefficient. This method differs from (\ref{Land}) in that a 
momentum term $\beta_n^\d (\xi_n^\d - \xi_{n-1}^\d)$ is added to define the iterates. The method 
(\ref{AHB.2}) has intimate connection with the heavy ball method. Indeed, when $\R(x) = \|x\|^2/2$
and $F$ is Fr\'{e}chet differentiable, the method (\ref{AHB.2}) with $L(x) = F'(x)$ becomes 
$$
x_{n+1}^\d = x_n^\d - \a_n^\d F'(x_n^\d)^* (F(x_n^\d) - y^\d) + \beta_n^\d (x_n^\d - x_{n-1}^\d)
$$
which can be obtained by applying the heavy ball method (\ref{HB0}) to the minimization problem 
(\ref{MP}) with 
\begin{align}\label{obj}
f(x) := \frac{1}{2} \|F(x) - y^\d\|^2. 
\end{align}
We will further demonstrate that, when $F$ is a bounded linear operator, the method (\ref{AHB.2}) 
with general strongly convex $\R$, can be derived by applying the heavy ball method to a dual problem. 
Based on the successful experience with the acceleration effect of the method (\ref{HB0}) for well-posed 
optimization problems, it is natural to expect that, with suitable choices of the parameters $\a_n^\d$ 
and $\beta_n^\d$, the method (\ref{AHB.2}) can have superior performance over the method (\ref{Land}) 
for ill-posed problems. In this paper we will provide a criterion for choosing $\a_n^\d$ and $\beta_n^\d$ 
and investigate the convergence property of the corresponding method. Our choices of $\a_n^\d$ and $\beta_n^\d$ 
are adaptive in the sense that only quantities arising from the computation are used. The formulae 
for calculating $\a_n^\d$ and $\beta_n^\d$ are explicit, thereby avoiding the need for a backtracking 
search strategy and consequently saving computational time. The complexity per iteration of our proposed 
method is comparable to one step of the Landweber-type method (\ref{Land}). However, extensive 
computational results demonstrate that our method significantly accelerates the Landweber-type method 
in terms of both the number of iterations and the computational time. 

This paper is organized as follows. In Section \ref{sect2} we first formulate the conditions on the 
regularization function $\R$ and the forward operator $F$, we then propose our adaptive heavy ball 
method by elucidating how to choose the step-sizes and the momentum coefficients, we also show that 
our method is well-defined and the iteration can be terminated in finite many steps by the discrepancy 
principle. In Section \ref{sect3} we focus on the convergence analysis and establish the regularization 
property of our method. In Section \ref{sect4} we perform various numerical simulations to demonstrate
the acceleration effect and superior performance of our method over the Landweber-type method. 

\section{\bf The adaptive heavy-ball method}\label{sect2}

In this section we will consider the iterative method (\ref{AHB.2}) and provide a motivation on 
choosing the step-size $\a_n^\d$ and the momentum coefficient $\beta_n^\d$ adaptively. We will then 
show that our proposed method, i.e. Algorithm \ref{alg:AHB} below, is well-defined. 

\subsection{\bf Assumptions}

Note that the method (\ref{AHB.2}) involves the regularization function $\R$ and the forward 
operator $F$. In order to carry out the analysis, certain conditions should be placed on $\R$ 
and $F$. For the regularization function $\R$, we assume the following standard condition. 

\begin{assumption}\label{ass0}
{\it $\R : X\to (-\infty, \infty]$ is a proper, lower semi-continuous, strongly convex function in 
the sense that there is a constant $\sigma>0$ such that
\begin{align*}
\R(t \bar x + (1-t) x) + \sigma t(1-t) \|\bar x -x\|^2 \le t \R(\bar x) + (1-t) \R(x)
\end{align*}
for all $\bar x, x\in \emph{dom}(\R)$ and $0\le t\le 1$.}
\end{assumption}

To make use of Assumption \ref{ass0}, let us recall a little convex analysis (\cite{Z2002}). For any 
proper convex function $\R: X \to (-\infty, \infty]$ we use $\p \R$ to denote its subdifferential, i.e. 
$$
\p \R(x) = \{\xi\in X: \R(z) \ge \p \R(x) + \l \xi, z - x\r \mbox{ for all } z \in X\}. 
$$
For any $x\in X$ with $\p \R(x) \ne \emptyset$, let $\xi \in \p \R(x)$ and define 
$$
D_\R^{\xi}(z,x) := \R(z) - \R(x) - \l \xi, z - x\r, \quad z \in X
$$
which is called the Bregman distance induced by $\R$ at $x$ in the direction $\xi$. If $\R$ is strongly 
convex in the sense of Assumption \ref{ass0}, it is easy to derive that 
\begin{align}\label{sc}
D_\R^{\xi}(z, x) \ge \sigma \| z - x\|^2
\end{align}
for all $x\in X$ with $\p \R(x) \ne \emptyset$, $\xi \in \p \R(x)$ and $z \in X$. For a proper, 
lower semi-continuous, convex function $\R: X \to (-\infty, \infty]$, its convex conjugate
$$
\R^*(\xi) := \sup_{x\in X} \left\{\l \xi, x\r - \R(x)\right\}, \quad \forall \xi \in X
$$
is also a proper, lower semi-continuous, convex function that maps from $X$ to $(-\infty, \infty]$. Moreover 
$$
\xi \in \p\R(x) \iff x \in \p\R^*(\xi) \iff \R(x) + \R^*(\xi) = \l \xi, x\r. 
$$
If $\R$ satisfies Assumption \ref{ass0}, then $\R^*$ is continuous differentiable, its gradient 
$\nabla \R^*$ is Lipschitz continuous with 
\begin{align}\label{RLip}
\|\nabla \R^*(\bar \xi) - \nabla \R^*(\xi)\| \le \frac{\|\bar \xi- \xi\|}{2 \sigma}, \quad 
\forall \bar \xi, \xi\in X; 
\end{align}
see \cite[Corollary 3.5.11]{Z2002}. Thus, under Assumption \ref{ass0}, for any $\xi\in X$ the minimization 
problem 
$$
x = \arg \min_{z\in X} \left\{\R(z) - \l \xi, z\r\right\}
$$
has a unique solution given by $x = \nabla \R^*(\xi)$. Clearly $\xi \in \p \R(x)$. 

For the forward operator $F$, we assume the following condition which has been widely used in 
dealing with ill-posed problems, where $B_{\rho}(x_0):=\{x\in X: \|x-x_0\|\leq \rho\}$ for 
any $\rho>0$.

\begin{assumption}\label{ass1}
{\it \begin{itemize}[leftmargin = 0.8cm]
\item[\emph{(a)}] $X$ and $Y$ are Hilbert spaces. 		
  
\item[\emph{(b)}] There exist $\rho>0$, $x_0 \in X$ and $\xi_0 \in \p \R(x_0)$ such that 
$B_{2\rho}(x_0)\subset \emph{dom}(F)$ and (\ref{AHB.1}) has a solution $\bar x$ such that 
$D_\R^{\xi_0} (\bar x, x_0) \le \sigma \rho^2$.
  
\item[\emph{(c)}] $F$ is weakly closed on $\emph{dom}(F)$, i.e. for any sequence 
$\{x_n\} \subset \emph{dom}(F)$ satisfying $x_n \rightharpoonup x \in X$ and $F(x_n)\to v \in Y$ 
there hold $x \in \emph{dom}(F)$ and $F(x) = v$. Throughout the paper ``$\to$' and 
``$\rightharpoonup$" are used to denote the strong and weak convergence respectively. 
  
\item[\emph{(d)}] There exist a family of bounded linear operators $\{L(x): X \to Y\}_{x\in B_{2\rho}(x_0)}$ 
such that $x \to L(x)$ is continuous on $B_{2\rho}(x_0)$ and there is $0\le \eta <1$ such that
$$
\|F(x) - F(\bar x) - L(\bar x)(x - \bar x)\|\le \eta \|F(x) - F(\bar x)\|
$$
for all $x,\bar x\in B_{2\rho}(x_0)$. Moreover, there is a constant $L>0$ such that $\|L(x)\| \le L$ for 
all $x \in B_{2\rho}(x_0)$. 
\end{itemize}
}
\end{assumption}

We remark that, when $F$ is a bounded linear operator, Assumption \ref{ass1} holds automatically with 
$\rho = \infty$ and $\eta = 0$. When $F$ is nonlinear, Assumption \ref{ass1} (d) is known as the 
tangential cone condition which gives certain restriction on the nonlinearity of $F$. Assumption 
\ref{ass1} (d) clearly implies 
$$
\|F(x) - F(\bar x) \| \le \frac{1}{1-\eta} \|L(\bar x) (x - \bar x)\|, \quad 
\forall x, \bar x \in B_{2\rho}(x_0)
$$
and thus the continuity of the mapping $x \to F(x)$ on $B_{2\rho}(x_0)$. Based on Assumption \ref{ass0}
and Assumption \ref{ass1}, it has been shown in \cite[Lemma 3.2]{JW2013} that (\ref{AHB.1}) has a 
unique solution $x^\dag$ satisfying 
\begin{align}\label{Rmin}
D_\R^{\xi_0}(x^\dag, x_0) := \min_{x\in \mbox{dom}(F)} \left\{D_\R^{\xi_0}(x, x_0): F(x) = y\right\}.
\end{align}
By using (\ref{sc}) and Assumption \ref{ass1} (b) it is easy to see that $\|x^\dag - x_0\| \le \rho$, 
i.e. $x^\dag \in B_\rho(x_0)$.

\subsection{\bf The method}

Let us now turn to the consideration on the method (\ref{AHB.2}). We first demonstrate that, 
when $F$ is a bounded linear operator, the method (\ref{AHB.2}) can be derived by applying 
the heavy ball method to the dual problem of 
\begin{align}
\min\{\R(x): x \in X \mbox{ and } F x = y\}.
\end{align}
To see this, we recall that the associated dual function is given by 
\begin{align*}
d(\la) := \inf_{x\in X} \left\{\R(x) - \l \la, F x - y \r\right\} 
= - \R^*(F^* \la) + \l \la, y\r
\end{align*}
and hence the corresponding dual problem is $\max_{\la \in Y} d(\la)$, i.e. 
\begin{align}\label{dual}
\min_{\la \in Y} \left\{\R^*(F^* \la) - \l \la, y\r\right\}.
\end{align}
Since $\R$ satisfies Assumption \ref{ass0}, the objective function in (\ref{dual}) is continuous 
differentiable. Applying the heavy ball method to solve (\ref{dual}) gives 
\begin{align*}
\la_{n+1} = \la_n - \a_n (F \nabla \R^*(F^* \la_n) - y) + \beta_n (\la_n - \la_{n-1}).
\end{align*}
By setting $\xi_n := F^* \la_n$ and $x_n := \nabla \R^*(F^* \la_n)$, we then obtain 
\begin{align*}
x_n = \nabla \R^*(\xi_n), \quad 
\xi_{n+1} = \xi_n - \a_n F^* (F x_n - y) + \beta_n (\xi_n - \xi_{n-1}).
\end{align*}
In case only noisy data $y^\d$ is available, by replacing $y$ by $y^\d$ and allowing $\a_n$, $\beta_n$ to 
be dependent on $y^\d$ in the above equation we may obtain the method 
\begin{align*}
x_n^\d = \nabla \R^*(\xi_n^\d), \quad 
\xi_{n+1}^\d = \xi_n^\d - \a_n^\d F^* (F x_n^\d - y^\d) + \beta_n^\d (\xi_n^\d - \xi_{n-1}^\d)
\end{align*}
which is equivalent to (\ref{AHB.2}) when $F$ is bounded linear. 

Note that the term $F^*(Fx_n^\d - y^\d)$ in the above equation is the gradient of the function 
$x\to \frac{1}{2}\|F x- y^\d\|^2$ at $x_n^\d$. When $F$ is a nonlinear Fr\'{e}chet differentiable operator, 
we may use the gradient $F'(x_n^\d)^*(F(x_n^\d) - y^\d)$ of the function (\ref{obj}) at $x_n^\d$ to 
replace $F^*(F x_n^\d - y^\d)$ in the above equation; in case $F$ is not differentiable at $x_n^\d$, 
we may use $L(x_n^\d)$ to substitute $F'(x_n^\d)$. Consequently, it leads to the method (\ref{AHB.2}) for solving (\ref{AHB.1}). 

In order for the method (\ref{AHB.2}) to have fast convergence, it is natural to choose $\a_n^\d$ and 
$\beta_n^\d$ at each iteration step such that $x_{n+1}^\d$ is as close to a solution of (\ref{AHB.1}) 
as possible. Let $\hat x$ denote any solution of (\ref{AHB.1}) in $B_{2\rho}(x_0)$. We may measure the 
closeness from $x_{n+1}^\d$ to $\hat x$ by the Bregman distance 
$$
D_\R^{\xi_{n+1}^\d}(\hat x, x_{n+1}^\d) := \R(\hat x) - \R(x_{n+1}^\d) 
- \l \xi_{n+1}^\d, \hat x - x_{n+1}^\d\r
$$
induced by $\R$ at $x_{n+1}^\d$ in the direction $\xi_{n+1}^\d$. Directly minimizing this quantity 
becomes infeasible because it involves an unknown solution $\hat x$ and a strongly convex function $\R$ 
which may be non-quadratic. As a compromise, we may first derive a suitable upper bound for this 
quantity and then take minimization steps. 

We will elucidate the idea on choosing $\a_n^\d$ and $\beta_n^\d$ below. For simplicity of 
exposition, we introduce the notation
$$
r_n^\d: = F(x_n^\d) - y^\d, \quad g_n^\d := L(x_n^\d)^* r_n^\d \quad \mbox{ and } \quad 
m_n^\d := \xi_n^\d - \xi_{n-1}^\d
$$
for all integers $n\ge 0$. Then  (\ref{AHB.2}) can be written as 
\begin{align}\label{AHB.2.0}
\xi_{n+1}^\d = \xi_n^\d - \a_n^\d g_n^\d + \beta_n^\d m_n^\d, \quad
x_{n+1}^\d = \nabla \R^*(\xi_{n+1}^\d). 
\end{align}
Let $\Delta_n^\d:= D_\R^{\xi_n^\d}(\hat x, x_n^\d)$ for $n\ge 0$. Since $\xi_n^\d \in \p \R(x_n^\d)$, 
we have 
$
\R(x_n^\d) + \R^*(\xi_n^\d) = \l\xi_n^\d, x_n^\d\r
$
and thus 
$$
\Delta_n^\d = \R(\hat x) + \R^*(\xi_n^\d) - \l \xi_n^\d, \hat x\r. 
$$
Therefore, by using $x_n^\d = \nabla \R^*(\xi_n^\d)$, (\ref{RLip}) and (\ref{AHB.2.0}), we can obtain
\begin{align*}
\Delta_{n+1}^\d - \Delta_n^\d 
& = \left( \R(\hat x) + \R^*(\xi_{n+1}^\d) - \l \xi_{n+1}^\d, \hat x\r\right) 
- \left( \R(\hat x) + \R^*(\xi_n^\d) - \l \xi_n^\d, \hat x\r\right) \\
& = \R^*(\xi_{n+1}^\d) - \R^*(\xi_n^\d) - \l \xi_{n+1}^\d - \xi_n^\d, \nabla\R^*(\xi_n^\d)\r 
+ \l \xi_{n+1}^\d - \xi_n^\d, x_n^\d - \hat x\r \\
& \le \frac{1}{4\sigma} \|\xi_{n+1}^\d - \xi_n^\d\|^2 + \l \xi_{n+1}^\d - \xi_n^\d, x_n^\d - \hat x\r \\
& =  \frac{1}{4\sigma} \|\a_n^\d g_n^\d - \beta_n^\d m_n^\d\|^2 
+ \l - \a_n^\d g_n^\d + \beta_n^\d m_n^\d, x_n^\d - \hat x\r.
\end{align*}
By the polarization identity in Hilbert spaces, we then have 
\begin{align*}
\Delta_{n+1}^\d - \Delta_n^\d 
& \le \frac{1}{4\sigma} \left( (\a_n^\d)^2 \|g_n^\d\|^2 + (\beta_n^\d)^2 \|m_n^\d\|^2 - 2 \a_n^\d \beta_n^\d \l g_n^\d, m_n^\d\r\right) \\
& \quad \, - \a_n^\d \l r_n^\d, L(x_n^\d)(x_n^\d-\hat x)\r + \beta_n^\d \l m_n^\d, x_n^\d - \hat x\r \\
& = \frac{1}{4\sigma} (\a_n^\d)^2 \|g_n^\d\|^2 + \frac{1}{4\sigma} (\beta_n^\d)^2 \|m_n^\d\|^2 
- \frac{1}{2\sigma} \a_n^\d \beta_n^\d \l g_n^\d, m_n^\d\r \\
& \quad \, - \a_n^\d \l r_n^\d, r_n^\d + y^\d - y + y - F(x_n^\d) - L(x_n^\d)(\hat x - x_n^\d)\r \\
& \quad \, + \beta_n^\d \l m_n^\d, x_n^\d - \hat x\r .
\end{align*}
Assume $x_n^\d \in B_{2\rho}(x_0)$. By using $\|y^\d - y\| \le \d$, the Cauchy-Schwarz inequality 
and Assumption \ref{ass1} (d), we further obtain 
\begin{align}\label{AHB.83}
\Delta_{n+1}^\d - \Delta_n^\d 
& \le \frac{1}{4\sigma} (\a_n^\d)^2 \|g_n^\d\|^2 + \frac{1}{4\sigma} (\beta_n^\d)^2 \|m_n^\d\|^2 
- \frac{1}{2\sigma} \a_n^\d \beta_n^\d \l g_n^\d, m_n^\d\r - \a_n^\d \|r_n^\d\|^2 \nonumber \\
& \quad \, + \d \a_n^\d \|r_n^\d\| + \eta \a_n^\d \|r_n^\d\| \|y - F(x_n^\d)\| 
+ \beta_n^\d \l m_n^\d, x_n^\d - \hat x\r \nonumber \\
& \le  \frac{1}{4\sigma} (\a_n^\d)^2 \|g_n^\d\|^2 - \a_n^\d \|r_n^\d\|^2 + (1+\eta) \d \a_n^\d \|r_n^\d\| 
+ \eta \a_n^\d \|r_n^\d\|^2 \nonumber \\
& \quad\, + \beta_n^\d \l m_n^\d, x_n^\d - \hat x\r + \frac{1}{4\sigma} (\beta_n^\d)^2 \|m_n^\d\|^2 
- \frac{1}{2\sigma} \a_n^\d \beta_n^\d \l g_n^\d, m_n^\d\r.
\end{align}
Note that when $\beta_n^\d = 0$, the method (\ref{AHB.2}) becomes the Landweber-type method (\ref{Land}) and 
extensive analysis has been done on (\ref{Land}) with suitable choices of the step-size $\a_n^\d$, 
including (\ref{css}) and (\ref{ass}). For the method (\ref{AHB.2}) we also choose $\a_n^\d$ as 
such, i.e.  
\begin{align}\label{alphan}
\a_n^\d := \frac{\mu_0}{L^2} \quad \mbox{ or } \quad 
\a_n^\d := \min\left\{\frac{\mu_0\|r_n^\d\|^2}{\|g_n^\d\|^2}, \mu_1\right\}
\end{align}
when $\|r_n^\d\| >\tau \d$, where $\mu_0$ and $\mu_1$ are two positive constants. Then, by plugging 
the choice of $\a_n^\d$ from (\ref{alphan}) into (\ref{AHB.83}) and using $\d \le \|r_n^\d\|/\tau$ 
we have 
\begin{align}\label{AHB.6}
\Delta_{n+1}^\d - \Delta_n^\d 
& \le - \left(1 - \frac{1+\eta}{\tau} - \eta - \frac{\mu_0}{4\sigma} \right) \a_n^\d \|r_n^\d\|^2 
+ \frac{1}{4\sigma} (\beta_n^\d)^2 \|m_n^\d\|^2 \nonumber \\
& \quad \, + \beta_n^\d \l m_n^\d, x_n^\d - \hat x\r  
- \frac{1}{2\sigma} \a_n^\d \beta_n^\d \l g_n^\d, m_n^\d\r. 
\end{align}
We next consider the choice of $\beta_n^\d$. It seems natural to obtain it by minimizing the right 
hand side of the above equation with respect to $\beta_n^\d$. Let us fix a number $\beta\in (0,\infty]$ 
and minimize the right hand side of (\ref{AHB.6}) over $[0, \beta]$ to obtain 
$$
\beta_n^\d = \min\left\{\max\left\{0, 
\frac{\a_n^\d \l g_n^\d, m_n^\d\r - 2\sigma \l m_n^\d, x_n^\d - \hat x\r}{\|m_n^\d\|^2}\right\}, 
\beta\right\}
$$
whenever $m_n^\d\ne 0$. However, this formula for $\beta_n^\d$ is not computable because 
it involves $\hat x$ which is unknown. We need to treat the term 
\begin{align}\label{AHB.3}
\gamma_n^\d := \l m_n^\d, x_n^\d - \hat x \r 
\end{align}
by using a suitable computable surrogate. Note that $\gamma_0^\d = 0$ and, for $n\ge 1$, 
\begin{align*}
\gamma_n^\d & = \l m_n^\d, x_n^\d - x_{n-1}^\d \r + \l m_n^\d, x_{n-1}^\d - \hat x\r \\
& = \l m_n^\d, x_n^\d - x_{n-1}^\d\r + \l -\a_{n-1}^\d g_{n-1}^\d 
+ \beta_{n-1}^\d m_{n-1}^\d, x_{n-1}^\d - \hat x\r \\
& = \l m_n^\d, x_n^\d - x_{n-1}^\d \r - \a_{n-1}^\d \l r_{n-1}^\d, L(x_{n-1}^\d)(x_{n-1}^\d-\hat x)\r 
+ \beta_{n-1}^\d \gamma_{n-1}^\d \\
& = \l m_n^\d, x_n^\d - x_{n-1}^\d \r - \a_{n-1}^\d \| r_{n-1}^\d\|^2 \\
& \quad \, - \a_{n-1}^\d \l r_{n-1}^\d, y^\d - F(x_{n-1}^\d) +  L(x_{n-1}^\d)(x_{n-1}^\d-\hat x)\r 
+ \beta_{n-1}^\d \gamma_{n-1}^\d \\
& \le \l m_n^\d, x_n^\d - x_{n-1}^\d \r - \a_{n-1}^\d \| r_{n-1}^\d\|^2 \\
& \quad \, + \a_{n-1}^\d \| r_{n-1}^\d\| 
\left(\d + \| y - F(x_{n-1}^\d) - L(x_{n-1}^\d)(\hat x - x_{n-1}^\d)\|\right) 
+ \beta_{n-1}^\d \gamma_{n-1}^\d.
\end{align*} 
Assuming $x_{n-1}^\d \in B_{2\rho}(x_0)$, we may use Assumption \ref{ass1} (d) to further obtain 
\begin{align}\label{AHB.4}
\gamma_n^\d 
& \le \l m_n^\d, x_n^\d - x_{n-1}^\d \r - \a_{n-1}^\d \| r_{n-1}^\d\|^2 
+ \a_{n-1}^\d \| r_{n-1}^\d\| \left(\d + \eta \| y - F(x_{n-1}^\d) \|\right) \nonumber \\
& \quad \, + \beta_{n-1}^\d \gamma_{n-1}^\d \nonumber \\
& \le \l m_n^\d, x_n^\d - x_{n-1}^\d \r - (1-\eta)\a_{n-1}^\d \|r_{n-1}^\d\|^2 \nonumber \\
& \quad \, + (1+\eta)\a_{n-1}^\d \d \|r_{n-1}^\d\| + \beta_{n-1}^\d \gamma_{n-1}^\d.
\end{align}
This motivates us to introduce $\{\tilde \gamma_n^\d\}$ such that 
\begin{align}\label{etan}
\tilde \gamma_n^\d = \left\{\begin{array}{lll} 
0 & \mbox{ if } n =0, \\
\l m_n^\d, x_n^\d - x_{n-1}^\d\r - (1-\eta) \a_{n-1}^\d \|r_{n-1}^\d\|^2 & \\
\qquad \qquad + (1+\eta) \a_{n-1}^\d \d \|r_{n-1}^\d\| + \beta_{n-1}^\d \tilde \gamma_{n-1}^\d & \mbox{ if } n \ge 1
\end{array}\right.
\end{align}
once $x_k^\d$ for $0\le k \le n$ and $\a_k^\d, \beta_k^\d$ for $0\le k \le n-1$ are defined.  
Therefore, it leads us to propose Algorithm \ref{alg:AHB} below that will be considered in this paper. 

\begin{algorithm}[AHB: Adaptive heavy ball method]\label{alg:AHB}
Take an initial guess $\xi_0$, calculate $x_0:= \arg\min_{x\in X} \{\R(x) - \l \xi_0, x\r\}$, and set $\xi_{-1}^\d = \xi_0^\d = \xi_0$. Pick the numbers $\tau>1$, $0<\beta\le \infty$, $\mu_0>0$ and $\mu_1>0$. For $n\ge 0$ do the following: 
		
\begin{enumerate}[leftmargin = 0.8cm]
\item[\emph{(i)}] Calculate $r_n^\d := F(x_n^\d) - y^\d $. If $\|r_n^\d\|\le \tau \d$, stop and output $x_n^\d$; 
			
\item[\emph{(ii)}] Calculate $g_n^\d := L(x_n^\d)^*(F(x_n^\d) - y^\d)$ and determine $\a_n^\d$ according to (\ref{alphan}); 
			
\item[\emph{(iii)}] Calculate $m_n^\d := \xi_n^\d - \xi_{n-1}^\d$ and determine $\tilde \gamma_n^\d$ by the formula (\ref{etan}); 
			
\item[\emph{(iv)}] Calculate $\beta_n^\d$ by 
\begin{align*}
\beta_n^\d = \left\{\begin{array}{lll}
\min\left\{\max\left\{0, \frac{\a_n^\d \l g_n^\d, m_n^\d\r - 2\sigma \tilde \gamma_n^\d}{\|m_n^\d\|^2}\right\}, \beta\right\} 
& \emph{ if } m_n^\d \ne 0,\\
0 & \emph{ if } m_n^\d =0; 
\end{array}\right.
\end{align*}
			
\item[\emph{(v)}] Update $\xi_{n+1}^\d$ and $x_{n+1}^\d$ by 
$$
\xi_{n+1}^\d = \xi_n^\d - \a_n^\d g_n^\d + \beta_n^\d m_n^\d, \quad 
x_{n+1}^\d  = \arg\min_{x\in X} \{\R(x) - \l \xi_{n+1}^\d, x\r \}.
$$
\end{enumerate}
\end{algorithm}

Note that, for the implementation of one step of Algorithm \ref{alg:AHB}, the most expensive parts are the 
calculations of $r_n^\d$, $g_n^\d$ and $x_{n+1}^\d$ which are common for Landweber-type method (\ref{Land});
the computational load for other parts relating to $\a_n^\d$, $\tilde \gamma_n^\d$ and $\beta_n^\d$ can be 
negligible. Therefore, the computational complexity per iteration of Algorithm \ref{alg:AHB} is marginally
higher than, but very close to, that of one step of the Landweber-type method (\ref{Land}). 

In the design of Algorithm \ref{alg:AHB}, the discrepancy principle has been incorporated, i.e. the iteration 
stops as long as $\|r_n^\d\|\le \tau \d$ is satisfied for the first time. In the following we will prove that 
Algorithm \ref{alg:AHB} is well-defined by showing that the iteration indeed can stop in finite many steps. 

\begin{lemma}\label{AHB.lem1}
Let Assumption \ref{ass0} and Assumption \ref{ass1} hold. Consider Algorithm \ref{alg:AHB} and let $k\ge 0$ 
be an integer such that $x_n^\d \in B_{2\rho}(x_0)$ for all $0\le n < k$. Define $\gamma_n^\d$ by 
(\ref{AHB.3}) for $0\le n\le k$. Then $\gamma_n^\d \le \tilde \gamma_n^\d$ for all $0\le n \le k$.  
\end{lemma}
	
\begin{proof}
Since $\gamma_0^\d = \tilde \gamma_0^\d =0$, the result is true for $n =0$. Assume next that 
$\gamma_n^\d \le \tilde \gamma_n^\d$ for some $0\le n < k$. By noting that $\beta_n^\d \ge 0$, 
we may use (\ref{AHB.4}) and the definition of $\tilde \gamma_{n+1}^\d$ to obtain 
\begin{align*}
\gamma_{n+1}^\d 
& \le \l m_{n+1}^\d, x_{n+1}^\d - x_n^\d\r - (1-\eta) \a_{n}^\d \|r_{n}^\d\|^2
  + (1+\eta) \a_{n}^\d \d \|r_{n}^\d\| + \beta_n^\d \gamma_n^\d \\
& \le \l m_{n+1}^\d, x_{n+1}^\d - x_n^\d\r - (1-\eta) \a_n^\d \|r_n^\d\|^2
+ (1+\eta) \a_n^\d \d \|r_n^\d\| + \beta_n^\d \tilde \gamma_n^\d \\
& = \tilde \gamma_{n+1}^\d. 
\end{align*}
By the induction principle, this shows the result. 
\end{proof}
	
\begin{lemma}\label{AHB.lem2}
Let Assumption \ref{ass0} and Assumption \ref{ass1} hold. Consider Algorithm \ref{alg:AHB} with $\mu_0>0$ 
and $\tau>1$ being chosen such that 
\begin{align}\label{mutau}
c_0: = 1-\frac{1+\eta}{\tau}-\eta-\frac{\mu_0}{4\sigma} > 0.
\end{align}
Let $k\ge 0$ be an integer such that $\|r_n^\d\| > \tau \d$ for all $0\le n < k$. Then 
$x_n^\d \in B_{2\rho}(x_0)$ for all $0\le n\le k$ and 
\begin{align*}
\Delta_{n+1}^\d \le \Delta_n^\d - c_0 \a_n^\d \|r_n^\d\|^2
\end{align*}
for all $0\le n < k$, where $\Delta_n^\d := D_\R^{\xi_n^\d}(\hat x, x_n^\d)$ and $\hat x$ denotes any 
solution of (\ref{AHB.1}) in $B_{2\rho}(x_0)]\cap \emph{dom}(\R)$.
\end{lemma}
	
\begin{proof}
We first show by induction that 
\begin{align}\label{AHB.80}
x_n^\d \in B_{2\rho}(x_0) \quad \mbox{and} \quad 
D_\R^{\xi_n^\d} (x^\dag, x_n^\d) \le D_\R^{\xi_0}(x^\dag, x_0)
\end{align}
for all integers $0\le n \le k$. It is trivial for $n = 0$ as $\xi_0^\d = \xi_0$ and $x_0^\d = x_0$.
Next we assume that (\ref{AHB.80}) holds for all $0\le n \le l$ for some $l< k$. According 
to (\ref{AHB.6}) and Lemma \ref{AHB.lem1} we have for any solution $\hat x$ of (\ref{AHB.1}) 
in $B_{2\rho}(x_0) \cap \mbox{dom}(\R)$ that
\begin{align*}
\Delta_{l+1}^\d - \Delta_l^\d  
& \le - c_0 \a_l^\d \|r_l^\d\|^2 + \frac{1}{4\sigma} (\beta_l^\d)^2 \|m_l^\d\|^2 
+ \beta_l^\d \gamma_l^\d  - \frac{1}{2\sigma} \a_l^\d \beta_l^\d \l g_l^\d, m_l^\d\r \\
& \le - c_0 \a_l^\d \|r_l^\d\|^2 + \frac{1}{4\sigma} (\beta_l^\d)^2 \|m_l^\d\|^2 
+ \beta_l^\d \tilde \gamma_l^\d  - \frac{1}{2\sigma} \a_l^\d \beta_l^\d \l g_l^\d, m_l^\d\r.
\end{align*}
Note that $\beta_l^\d$ is the minimizer of the function $t \to h_l(t)$ over $[0, \beta]$, where 
$$
h_l(t) := - c_0 \a_l^\d \|r_l^\d\|^2 + \frac{1}{4\sigma} t^2 \|m_l^\d\|^2 
+ t \tilde \gamma_l^\d  - \frac{1}{2\sigma} \a_l^\d t \l g_l^\d, m_l^\d\r,
$$
we can conclude that 
\begin{align}\label{AHB.82}
\Delta_{l+1}^\d - \Delta_l^\d \le h_l(\beta_l^\d) \le h_l(0) = - c_0 \a_l^\d \|r_l^\d\|^2.
\end{align}
By virtue of this inequality with $\hat x = x^\dag$, the induction hypothesis, and Assumption \ref{ass1} (b) 
we have 
$$
D_\R^{\xi_{l+1}}(x^\dag, x_{l+1}^\d) \le D_\R^{\xi_l^\d}(x^\dag, x_l^\d) 
\le D_\R^{\xi_0}(x^\dag, x_0) \le \sigma \rho^2
$$
which together with (\ref{sc}) implies that $\sigma \|x_{l+1}^\d - x^\dag\|^2 \le \sigma \rho^2$ 
and hence $\|x_{l+1}^\d - x^\dag\| \le \rho$. Since $\|x^\dag - x_0\| \le \rho$, we thus 
have $\|x_{l+1}^\d - x_0\| \le 2\rho$, i.e. $x_{l+1}^\d \in B_{2\rho}(x_0)$. We therefore complete the 
proof of (\ref{AHB.80}). As a direct consequence, we can see that 
(\ref{AHB.82}) holds for all $0\le l < k$ which shows the desired result. 
\end{proof}

Based on Lemma \ref{AHB.lem2} we now can prove that Algorithm \ref{ass1} is well-defined, i.e. iteration 
must terminate in finite many steps. 

\begin{theorem}\label{AHB:thm0}
Let Assumption \ref{ass0} and Assumption \ref{ass1} hold. Consider Algorithm \ref{alg:AHB} with noisy data 
$y^\d$ satisfying $\|y^\d - y\| \le \d$ with noise level $\d >0$. Assume that $\tau > 1$ and $\mu_0 > 0$ 
are chosen such that (\ref{mutau}) holds. Then the algorithm must terminate in finite many steps, i.e. 
there exists a finite integer $n_\d$ such that 
\begin{align}\label{DP}
\|F(x_{n_\d}^\d) - y^\d\| \le \tau \d < \|F(x_n^\d) - y^\d\|, \quad 0\le n< n_\d. 
\end{align}
\end{theorem}
	
\begin{proof}
Let $l\ge 0$ be an integer such that $\|F(x_n^\d) - y^\d\|>\tau \d$ for all $0\le n \le l$. It then follows from Lemma \ref{AHB.lem2} that 
\begin{align*}
c_0 \sum_{n=0}^l \a_n^\d \|F(x_n^\d) - y^\d\|^2 
& \le \sum_{n=0}^l \left(\Delta_n^\d - \Delta_{n+1}^\d \right) \\
& = \Delta_0^\d - \Delta_{l+1}^\d \le \Delta_0^\d \\
& = D_\R^{\xi_0} (\hat x, x_0). 
\end{align*}
According to the choice of $\a_n^\d$ and $\|L(x)\|\leq L$ for all $x\in B_{2\rho}(x_0)$, we can see $\a_n^\d \ge \min\{\mu_0/L^2, \mu_1\}$ for all $0\le n\le l$. Thus 
\begin{align*}
(l+1) c_0 \min\{\mu_0/L^2, \mu_1\} \tau^2 \d^2 \le c_0 \sum_{n=0}^l \a_n^\d \|F(x_n^\d) - y^\d\|^2 
\le D_\R^{\xi_0}(\hat x, x_0) < \infty. 
\end{align*}
If there is no finite integer $n_\d$ such that (\ref{DP}) holds, then we can take $l$ to be any positive integer. Letting $l \to \infty$ in the above equation gives a contradiction. Thus, the algorithm must terminate in finite many steps. 
\end{proof}

\section{\bf Convergence analysis}\label{sect3}

In this section we will provide convergence analysis on Algorithm \ref{alg:AHB}. Let $n_\d$ be the integer 
output by the discrepancy principle, we will analyze the behavior of $x_{n_\d}^\d$ as $\d \to 0$. We first 
prove a weak convergence result as stated below.
	
\begin{theorem}\label{AHB:thm1}
Let Assumption \ref{ass0} and Assumption \ref{ass1} hold. Consider Algorithm \ref{alg:AHB} using a 
sequence of noisy data $\{y^{\d_k}\}$ satisfying $\|y^{\d_k} - y\| \le \d_k$ with $0<\d_k \to 0$ 
as $k\to \infty$. Assume that $\tau > 1$ and $\mu_0 > 0$ are chosen such that (\ref{mutau}) holds. Let 
$n_k := n_{\d_k}$ be the output integer. Then there is a subsequence $\{y^{\d_{k_l}}\}$ of 
$\{y^{\d_k}\}$ such that 
$$
x_{n_{k_l}}^{\d_{k_l}} \rightharpoonup x^* \quad \mbox{ as } l \to \infty
$$
for some solution $x^*$ of (\ref{AHB.1}) in $B_{2\rho}(x_0)$.

If in addition $\R(x) = \|x\|^2/2$ and $\emph{Ker}(L(x^\dag))\subset\emph{Ker}(L(x))$ for 
all $x\in B_{2\rho}(x_0)$, then $x_{n_k}^{\d_k} \rightharpoonup x^\dag$ as $k \to \infty$, 
where $x^\dag$ is the unique $x_0$-minimum norm solution of (\ref{AHB.1}). 
\end{theorem}
	
\begin{proof}
By using Lemma \ref{AHB.lem2} with $\hat x = x^\dag$ and (\ref{sc}), we have 
$$
\sigma \|x_{n_k}^{\d_k} - \hat x\|^2 \le D_\R^{\xi_{n_k}^{\d_k}} (\hat x, x_{n_k}^{\d_k}) 
\le D_\R^{\xi_0}(\hat x, x_0)
$$ 
which implies that $\{x_{n_k}^{\d_k}\}$ is a bounded sequence. Thus, we can find a subsequence 
$\{y^{\d_{k_l}}\}$ of $\{y^{\d_l}\}$ such that $x_{n_{k_l}}^{\d_{k_l}} \rightharpoonup x^*$ 
as $l\to \infty$ for some $x^* \in X$. Because $\|F(x_{n_{k_l}}^{\d_{k_l}}) - y^{\d_{k_l}}\|
\le \tau \d_{k_l}$, we have $F(x_{n_{k_l}}^{\d_{k_l}}) \to y$ as $l\to \infty$. Thus, by the 
weak closedness of $F$ given in Assumption \ref{ass1} (c), we have $x^* \in \mbox{dom}(F)$ 
and $F(x^*) = y$. Furthermore, by the weak lower semi-continuity of norms we have 
\begin{align*}
\|x^* - x^\dag\|^2 \le \liminf_{l\to\infty} \|x_{n_{k_l}}^{\d_{k_l}} - x^\dag\|^2 
\le \frac{1}{\sigma} D_\R^{\xi_0}(x^\dag, x_0) \le \rho^2
\end{align*}
which together with $x^\dag \in B_\rho(x_0)$ implies that $x^* \in B_{2\rho}(x_0)$. 

When $\R(x) = \|x\|^2/2$ and $\mbox{Ker}(L(x^\dag))\subset\mbox{Ker}(L(x))$ for all $x\in B_{2\rho}(x_0)$, 
we show that  $x_{n_k}^{\d_k} \rightharpoonup x^\dag$ as $k \to \infty$, It suffices to show that 
$x^\dag$ is the unique weak cluster point of $\{x_{n_k}^{\d_k}\}$. For the given $\R$ we have 
$\xi_n^{\d_k} = x_n^{\d_k}$ for all $n$ and by the definition of $\xi_n^{\d_k}$ we can see that 
\begin{align*}
x_n^{\d_k} - x_0 
&\in \mbox{Ran}(L(x_{n-1}^{\d_k})^*) \oplus \cdots \oplus \mbox{Ran}(L(x_0^{\d_k})^*) \\
&\subset  \mbox{Ker}(L(x_{n-1}^{\d_k}))^\perp \oplus \cdots \oplus \mbox{Ker}(L(x_0^{\d_k}))^\perp \\
&\subset  \mbox{Ker}(L(x^\dag))^\perp
\end{align*}
for all $n$ which in particular implies that $x_{n_k}^{\d_k} - x_0 \in \mbox{Ker}(L(x^\dag))^\perp$. 
Consequently, for any weak cluster point $\bar x$ of  $\{x_{n_k}^{\d_k}\}$ we have 
$\bar x - x_0 \in \mbox{Ker}(L(x^\dag))^\perp$. Since $x^\dag$ is the $x_0$-minimum norm solution
of (\ref{AHB.1}), we must have $x^\dag - x_0 \in \mbox{Ker}(L(x^\dag))^\perp$. Indeed, for any 
$w \in \mbox{Ker}(L(x^\dag))$ we have $x^\dag + t w \in B_{2\rho}(x_0)$ for small $|t|$. Thus, it 
follows from Assumption \ref{ass1} (d) that 
\begin{align*}
\|F(x^\dag + t w) - F(x^\dag)\| 
& = \|F(x^\dag + t w) - F(x^\dag) - L(x^\dag) (t w)\| \\
&\le \eta \|F(x^\dag + t w) - F(x^\dag)\|
\end{align*}
which implies $F(x^\dag + t w) = F(x^\dag) = y$ for small $|t|$ as $0\le \eta <1$. Since $x^\dag$ is the 
$x_0$-minimal norm solution of (\ref{AHB.1}), we have $\|x^\dag + t w - x_0\|^2 \ge \|x^\dag - x_0\|^2$ 
for small $|t|$ and thus $\l x^\dag - x_0, w\r = 0$ for all $w \in \mbox{Ker}(L(x^\dag))$. Therefore 
$\bar x - x^\dag = (\bar x - x_0) + (x_0 - x^\dag) \in \mbox{Ker}(L(x^\dag))^\perp$

On the other hand, since both $\bar x$ and $x^\dag$ are solutions of $F(x) = y$ in $B_{2\rho}(x_0)$, 
we may use Assumption \ref{ass1} (d) to obtain 
\begin{align*}
\|L(x^\dag)(\bar x - x^\dag) \| \le (1+\eta) \|F(\bar x) - F(x^\dag)\| = 0
\end{align*}
which means $\bar x - x^\dag \in  \mbox{Ker}(L(x^\dag))$. Therefore 
$$
\bar x - x^\dag \in \mbox{Ker}(L(x^\dag))^\perp \cap  \mbox{Ker}(L(x^\dag)) = \{0\}
$$ 
and thus $\bar x = x^\dag$. The proof is complete. 
\end{proof}
	
The above theorem gives a weak convergence result on Algorithm \ref{alg:AHB} for any 
$\beta \in (0, \infty]$. However, it remains uncertain if a strong convergence result can 
be assured for general $\beta$. Nevertheless, by confining $\beta$ to be less than 1, we 
can establish a strong convergence result for Algorithm \ref{alg:AHB} in the remaining part 
of this section. The proof of strong convergence is rather challenging. We must explore the 
counterpart of Algorithm \ref{alg:AHB} utilizing the exact data $y$, which can be formulated 
as follows.

\begin{algorithm}\label{alg:AHB0}
Take an initial guess $\xi_0$, calculate $x_0 := \arg\min_{x\in X} \{\R(x) - \l \xi_0, x\r\}$, and set 
$\xi_{-1} = \xi_0$. Pick the numbers $0<\beta\le \infty$, $\mu_0>0$ and $\mu_1>0$. For $n\ge 0$ do 
the following: 
		
\begin{enumerate}[leftmargin = 0.8cm]
\item[\emph{(i)}] Calculate $r_n := F(x_n) - y$ and $g_n := L(x_n)^* r_n$;
			
\item[\emph{(ii)}] Determine $\a_n$ according to 
\begin{align*}
\a_n := \frac{\mu_0}{L^2} \quad \mbox{or } \quad 
\a_n := \left\{\begin{array}{lll}
\min\left\{\frac{\mu_0\|r_n\|^2}{\|g_n\|^2}, \mu_1\right\} & \emph{if } r_n \ne 0,\\
0 & \emph{if } r_n =0;
\end{array} \right.
\end{align*}
			
\item[\emph{(iii)}] Calculate $m_n := \xi_n - \xi_{n-1}$ and determine $\tilde \gamma_n$ by 
\begin{align*}
 \tilde \gamma_n = \left\{\begin{array}{lll} 
0 & \emph{if } n =0, \\
\l m_n, x_n - x_{n-1} \r - (1-\eta)\a_{n-1} \|r_{n-1}\|^2 + \beta_{n-1} \tilde \gamma_{n-1} & \emph{if } n \ge 1;
\end{array}\right.
\end{align*}
			
\item[\emph{(iv)}] Calculate $\beta_n$ by 
\begin{align*}
\beta_n = \left\{\begin{array}{lll}
\min\left\{\max\left\{0, 
\frac{\a_n \l g_n, m_n\r - 2\sigma \tilde \gamma_n}{\|m_n\|^2}\right\}, \beta\right\} & \emph{ if } m_n \ne 0,\\
0 & \emph{ if } m_n =0; 
\end{array}\right.
\end{align*}
			
\item[\emph{(v)}] Update $\xi_{n+1}$ and $x_{n+1}$ by 
$$
\xi_{n+1} = \xi_n - \a_n g_n + \beta_n m_n, \quad 
x_{n+1} = \arg\min_{x\in X} \{\R(x) - \l \xi_{n+1}, x\r\}. 
 $$
\end{enumerate}
\end{algorithm}
	
For the iterative sequence $\{x_n\}$ obtained by Algorithm \ref{alg:AHB0}, we first have the 
following result. 

\begin{lemma}\label{AHB.lem3}
Let Assumption \ref{ass0} and Assumption \ref{ass1} hold and consider Algorithm \ref{alg:AHB0} with 
$0<\mu_0<4\sigma (1-\eta)$. Then $x_n \in B_{2\rho}(x_0)$ for all integers $n\ge 0$ and for any solution 
$\hat x$ of (\ref{AHB.1}) in $B_{2\rho}(x_0)\cap \emph{dom}(\R)$ there holds
\begin{align*}
\Delta_{n+1} \le \Delta_n - c_1 \a_n\|F(x_n)-y\|^2, \quad\forall n\ge 0,
\end{align*}
where $\Delta_n := D_\R^{\xi_n}(\hat x, x_n)$ and $c_1: = 1 -\eta - \mu_0/(4\sigma) > 0$. Consequently 
$\{\Delta_n\}$ is monotonically decreasing and 
$$
\sum_{n=0}^\infty \a_n \|F(x_{n}) - y\|^2  < \infty. 
$$
\end{lemma}
	
\begin{proof}
By the similar argument in the proof of Lemma \ref{AHB.lem2} we can obtain the result immediately.
\end{proof}

Lemma \ref{AHB.lem3} has the following consequences which will be used later. 

\begin{lemma}\label{AHB.lem7}
Let Assumption \ref{ass0} and Assumption \ref{ass1} hold and consider Algorithm \ref{alg:AHB0} with $0<\mu_0<4\sigma(1-\eta)$.
	
\begin{enumerate}[leftmargin = 0.8cm]
\item[\emph{(i)}] If $F(x_n) - y=0$ for some $n\ge 0$, then $x_m = x_n$ for all $m > n$.
		
\item[\emph{(ii)}] If $\xi_{n+1} = \xi_n$ for some $n \ge 0$, then $F(x_n) - y=0$, $x_m = x_n$ and 
$\xi_m = \xi_n$ for all $m > n$. 
\end{enumerate}
\end{lemma}

\begin{proof}
(i) If $F(x_n) - y=0$ for some $n$, then $x_n$ is a solution of (\ref{AHB.1}) in $B_{2\rho}(x_0)\cap \mbox{dom}(\R)$. 
Thus we may use Lemma \ref{AHB.lem3} with $\hat x = x_n$ to conclude $D_\R^{\xi_m}(x_m, x_n) 
\le D_\R^{\xi_n}(x_n, x_n) = 0$ for all $m > n$. Consequently, it follows from (\ref{sc}) that 
$x_m = x_n$ for all $m > n$.
	
(ii) If $\xi_{n+1} = \xi_n$ for some $n$, then $x_{n+1} = \nabla \R^*(\xi_{n+1}) = \nabla \R^*(\xi_n) 
= x_n$. We thus have from Lemma \ref{AHB.lem3} that 
\begin{align*}
D_\R^{\xi_n}(x_n, \hat x)  = D_\R^{\xi_{n+1}}(x_{n+1}, \hat x) 
\le D_\R^{\xi_n}(x_n, \hat x) - c_1  \a_n \|F(x_n) - y\|^2
\end{align*}
which implies $\a_n \|F(x_n) - y\|^2 =0$. If $F(x_n) \ne y$, then the choice of $\a_n$ implies 
$\a_n>0$ and thus $\a_n \|F(x_n) - y\|^2>0$ which is a contradiction. Therefore we must have 
$F(x_n) = y$. By using (i) we then have $x_m = x_n$ for all $m \ge n$. Using these facts and an 
induction argument we can conclude that $\xi_m = \xi_n$ for all $m > n$. Indeed, this is trivial 
for $m = n+1$ by the given condition. Assume next $\xi_m = \xi_n$ for all $n < m \le k$ for
some $k>n$. Then, by the definition of $\xi_{k+1}$, we have 
\begin{align*}
\xi_{k+1} & = \xi_k - \a_k L(x_k)^*(F(x_k) - y) + \beta_k(\xi_k - \xi_{k-1})\\
& = \xi_n - \a_k L(x_n)^*(F(x_n) - y) = \xi_n. 
\end{align*}
By the induction principle, we thus complete the proof. 
\end{proof}

In order to show the strong convergence of the iterative sequence $\{x_n\}$ generated by 
Algorithm \ref{alg:AHB0}, we need the following result; see \cite[Proposition 2.3]{JL2014} 
or \cite[Proposition 3.6]{JW2013}.

\begin{proposition} \label{general}
Let Assumption \ref{ass0} and Assumption \ref{ass1} hold and consider the equation (\ref{AHB.1}). 
Let $\{x_n\}\subset B_{2\rho}(x_0) \cap \emph{dom}(\R)$ and $\{\xi_n\}\subset X$ be such that 

\begin{enumerate}[leftmargin = 0.8cm]
\item[\emph{(i)}] $\xi_n\in \p \R(x_n)$ for all $n$;

\item[\emph{(ii)}] for any solution $\hat x$ of (\ref{AHB.1}) in $B_{2\rho}(x_0) \cap \emph{dom}(\R)$ 
the sequence $\{D_\R^{\xi_n}(\hat x, x_n)\}$ is convergent;

\item[\emph{(iii)}] $\lim_{n\to \infty} \| F(x_n)-y\|=0$.

\item[\emph{(iv)}] there is a subsequence $\{n_k\}$ of integers with $n_k\rightarrow \infty$ such 
that for any solution $\hat x$ of (\ref{AHB.1}) in $B_{2\rho}(x_0) \cap \emph{dom}(\R)$ there holds
\begin{equation}\label{AHB.60}
 \lim_{l\rightarrow \infty} \sup_{k\ge l} |\l \xi_{n_k}-\xi_{n_l}, x_{n_k}-\hat x\r| =0.
\end{equation}
\end{enumerate}
Then there exists a solution $x^*$ of (\ref{AHB.1}) in $B_{2\rho}(x_0) \cap \emph{dom}(\R)$ such that
\begin{equation*}
\lim_{n\rightarrow \infty} D_\R^{\xi_n}(x^*, x_n) = 0.
\end{equation*}
If, in addition, $\xi_{n+1}-\xi_n \in \emph{Ker}(L(x^\dag))^\perp$ for all $n$, then $x^*=x^\dag$.
\end{proposition}

Based on Lemma \ref{AHB.lem3}, Lemma \ref{AHB.lem7} and Proposition \ref{general}, we are now ready 
to show the strong convergence of Algorithm \ref{alg:AHB0}. 
 
\begin{theorem}\label{AHB.lem5}
Let Assumption \ref{ass0} and Assumption \ref{ass1} hold and let $\{x_n\}$ be defined by Algorithm 
\ref{alg:AHB0} using the exact data $y$, where $0<\mu_0< 4\sigma(1-\eta)$ and $0<\beta <1$. Then 
there exists a solution $x^*$ of (\ref{AHB.1}) in $B_{2\rho}(x_0)$ such that 
\begin{align}\label{AHB.62}
\lim_{n\to \infty} D_\R^{\xi_n}(x^*, x_n) =0 \quad \mbox{ and } \quad 
\lim_{n\to \infty} \|x_n - x^*\| = 0.
\end{align}
If, in addition, $\emph{Ker}(L(x^\dag))\subset\emph{Ker}(L(x))$ for all $x\in B_{2\rho}(x_0)$, then 
$x^* = x^\dag$.
\end{theorem}
	
\begin{proof}
We will use Proposition \ref{general} to prove the conclusion. The item (i) in Proposition \ref{general} 
holds automatically by the definition of $x_n$. According to Lemma \ref{AHB.lem3} we immediately obtain the item (ii) in Proposition \ref{general}. To show item (iii), we note from the choice 
of $\{\a_n\}$ that $\a_n \ge c_3:= \min\{\mu_0/L^2,\mu_1\}$ when $r_n \ne 0$. Thus 
$\a_n \|r_n\|^2 \ge c_3 \|r_n\|^2$ for all $n\ge 0$.  From Lemma \ref{AHB.lem3} it then follows 
that
$$
c_3 \sum_{n=0}^\infty \|r_n\|^2 \le \sum_{n=0}^\infty \a_n \|r_n\|^2 <\infty. 
$$
This in particular implies (iii) in Proposition \ref{general}, i.e. 
\begin{align}\label{residualn}
\lim_{n\to\infty}\| r_n \|=0.
\end{align}
We now verify (iv) in Proposition \ref{general}. By using Lemma \ref{AHB.lem7} (i) we can conclude that 
\begin{align}\label{residualm}
\| r_n \|=0 \mbox{ for some } n \Longrightarrow \| r_m \|=0 \mbox{ for all } m\geq n.
\end{align} 
In view of (\ref{residualn}) and (\ref{residualm}), we can pick a sequence $\{n_{k}\}$ of 
integers by setting $n_0=0$ and letting $n_{k}$, for each $k\geq 1$, be the first integer satisfying
$$
n_{k}\geq n_{k-1}+1\quad \text{ and }\quad \| r_{n_{k}}\|\leq \|r_{n_{k-1}}\|.
$$
It is easy to see that
\begin{align}
\| r_{n_{k}}\|\leq \| r_n\|,\qquad 0\leq n\leq n_{k}.\label{rnk}
\end{align} 
For the above chosen $\{n_{k}\}$, we now show (\ref{AHB.60}) for any solution $\hat x$ of (\ref{AHB.1}) 
in $B_{2\rho}(x_0) \cap \mbox{dom}(\R)$. For any integers $l<k$ we write 
\begin{align}\label{eq1}
\l \xi_{n_k} - \xi_{n_l},x_{n_k} - \hat x\r 
= \sum_{n=n_l}^{n_k-1} \l \xi_{n+1} - \xi_n, x_{n_k} - \hat x \r. 
\end{align}
From the definition of $\{\xi_n\}$ and an induction argument it follows readily that 
\begin{align*}
\xi_{n+1} - \xi_n = - \sum_{i=0}^n \left(\prod_{j = i+1}^n \beta_j\right) \a_i L(x_i)^* r_i.
\end{align*}
Thus 
\begin{align*}
\l \xi_{n+1} - \xi_n, x_{n_k} - \hat x\r 
& = -\sum_{i=0}^n \left(\prod_{j=i+1}^n \beta_j\right) \a_i \l L(x_i)^* r_i, x_{n_k} - \hat x\r\\
& = - \sum_{i=0}^n \left(\prod_{j=i+1}^n \beta_j\right) \a_i\l r_i, L(x_i)(x_{n_k} - \hat x) \r. 
\end{align*}
Using Assumption \ref{ass1} (d) it is easy to obtain
\begin{align*}
\|L(x_i)(x_{n_k} - \hat x)\|
&\le \|L(x_i)(x_{i} - \hat x)\| + \|L(x_i)(x_{n_k} - x_{i})\|\\
&\le (1+\eta)\left( \|r_i\| + \|F(x_i) - F(x_{n_k})\| \right)\\
&\le (1+\eta)\left( 2\|r_i\| + \|r_{n_k}\|\right).
\end{align*}
Therefore, by using the Cauchy-Schwarz inequality, (\ref{rnk}), and $0\le \beta_n\le \beta$ 
for any $n\geq 0$, we can obtain
\begin{align*}
|\l \xi_{n+1} - \xi_{n}, x_{n_{k}} - \hat x\r| 
& \le \sum_{i=0}^n \left(\prod_{j=i+1}^n \beta_j\right) \a_i \| r_i\| \|L(x_i)(x_{n_k} - \hat x)\| \\
& \le (1+\eta)\sum_{i=0}^n \left(\prod_{j=i+1}^n \beta_j\right) \a_i\|r_i\| \left( 2\|r_i\| + \|r_{n_k}\|\right)  \\
& \le 3(1+\eta)\sum_{i=0}^n \left(\prod_{j=i+1}^n \beta_j\right) \a_i \|r_i\|^2 \\
& \le 3(1+\eta)\sum_{i=0}^n \beta^{n-i} \a_i \|r_i\|^2. 
\end{align*}
This together with (\ref{eq1}) then gives  
\begin{align}\label{AHB.61}
|\l \xi_{n_k} - \xi_{n_l},x_{n_k} - \hat x\r |
\le 3(1+\eta)\sum_{n=n_l}^{n_k-1}\sum_{i=0}^n \beta^{n-i} \a_i \|r_i\|^2.
\end{align}
To proceed further, let 
$$
S_N :=\sum_{n=0}^N\sum_{i=0}^{n}\beta^{n-i}\a_{i}\|r_{i}\|^2, \quad \forall N \ge 0. 
$$
Then 
\begin{align*}
S_{N+1} - S_N 
&= \sum_{n=0}^{N+1}\sum_{i=0}^n \beta^{n-i} \a_i \|r_i\|^2
- \sum_{n=0}^N \sum_{i=0}^n \beta^{n-i} \a_i\|r_i\|^2\\
& =\sum_{i=0}^{N+1} \beta^{N+1-i} \a_i\|r_i\|^2 \ge 0.
\end{align*}
Moreover, by rearranging the terms in $S_N$, utilizing the property $0<\beta <1$, and 
invoking Lemma \ref{AHB.lem3}, we can obtain
$$
S_N = \sum_{i=0}^N \left(\sum_{n=i}^N \beta^{n-i} \right)\a_i\|r_i\|^2 
\le \frac{1}{1-\beta} \sum_{i=0}^N \a_i\|r_i\|^2 
\le \frac{1}{1-\beta} \sum_{i=0}^\infty \a_i\|r_i\|^2 <\infty.
$$
Therefore, $\{S_N\}$ is a monotonically increasing sequence with a finite upper bound, and 
thus it converges by the monotone convergence theorem. In particular, $\{S_N\}$ is a Cauchy 
sequence and hence 
$$
\lim_{l \to \infty} \sup_{k \ge l} |S_{n_k-1} - S_{n_l-1}| = 0. 
$$
Note that 
$$
|S_{n_k-1}-S_{n_l-1}| = \sum_{n=n_l}^{n_k-1} \sum_{i=0}^n \beta^{n-i}\a_i \|r_i\|^2. 
$$
We thus have from (\ref{AHB.61}) that 
\begin{align*}
\sup_{k\ge l} |\l \xi_{n_k} - \xi_{n_l},x_{n_k} - \hat x\r | 
\le 3(1+\eta) \sup_{k\ge l} |S_{n_k-1}-S_{n_l-1}| \to 0
\end{align*}
as $l\to \infty$. This shows (iv) in Proposition \ref{general}. Thus, we may use Proposition \ref{general} 
to conclude the existence of a solution $x^*$ of (\ref{AHB.1}) in $B_{2\rho}(x_0) \cap \mbox{dom}(\R)$ such 
that (\ref{AHB.62}) holds. 
			
If $\mbox{Ker}(L(x^\dag))\subset\mbox{Ker}(L(x))$ for all $x\in B_{2\rho}(x_0)$, then we may use the 
definition of $\xi_{n+1}$ to obtain
\begin{align*}
\xi_{n+1} - \xi_n 
&\in \mbox{Ran}(L(x_n)^*) \oplus \cdots \oplus \mbox{Ran}(L(x_0)^*) \\
&\subset  \mbox{Ker}(L(x_n))^\perp \oplus \cdots \oplus \mbox{Ker}(L(x_0))^\perp \\
&\subset  \mbox{Ker}(L(x^\dag))^\perp. 
\end{align*}
Therefore, we may use the last part of Proposition \ref{general} to conclude $x^* = x^\dag$.
\end{proof}

In order to establish the regularization property of Algorithm \ref{alg:AHB}, we need a stability 
result to connect Algorithm \ref{alg:AHB} with Algorithm \ref{alg:AHB0}. The following resut is 
sufficient for our purpose. 

\begin{lemma}\label{AHB.lem6}
Let Assumption \ref{ass0} and Assumption \ref{ass1} hold, let $\tau>1$ and $\mu_0>0$ be chosen 
such that (\ref{mutau}) holds, and let $\beta<\infty$. Let $\{y^{\d_l}\}$ be a sequence of 
noisy data satisfying $\|y^{\d_l} - y\| \le \d_l$ with $0 < \d_l \to 0$ as $l \to \infty$,
and let $\xi_n^{\d_l}, x_n^{\d_l}$, $0\le n \le n_{\d_l}$, be defined by Algorithm \ref{alg:AHB}
using noisy data $y^{\d_l}$, where $n_{\d_l}$ denotes the output integer. Let $\{\xi_n, x_n\}$ be 
defined by the counterpart of Algorithm \ref{alg:AHB} using the exact data $y$. Then, for any 
finite integer $\hat n \le \liminf_{l \to \infty} n_{\d_l}$ there hold
$$
\xi_n^{\d_l} \to \xi_n \quad \mbox{ and } \quad x_n^{\d_l}\to x_n \quad \text{ as } l \to \infty
$$
for all $0\le n \le \hat n$. 
\end{lemma}
		
\begin{proof}
Since $\hat n \le \liminf_{l\to \infty} n_{\d_l}$, we can conclude that $n_{\d_l} \ge \hat n$ for 
large $l$ and hence $\xi_n^{\d_l}, x_n^{\d_l}$ are well-defined for all $0\le n \le \hat n$. Let 
\begin{align*}
n_* := \inf\{n: n\in {\mathbb N} \mbox{ and } \xi_n = \xi_{n-1} \},
\end{align*}
where ${\mathbb N}:=\{1, 2, \cdots\}$ denotes the set of natural numbers. Then $1 \le n_* \le \infty$. 
By the definition of $n_*$ we have $m_n \ne 0$ for $1\le n < n_*$ and by using Lemma \ref{AHB.lem7}
we also have $m_n =0$ for $n_* \le n <\infty$. 

We first use an induction argument to show that 
\begin{align}\label{AHB.39}
\xi_n^{\d_l} \to \xi_n, \quad x_n^{\d_l} \to x_n, \quad \a_n^{\d_l} r_n^{\d_l} \to \a_n r_n, \quad 
\beta_n^{\d_l} \to \beta_n, \quad \tilde \gamma_n^{\d_l} \to \tilde \gamma_n
\end{align}
as $l \to \infty$ for every $0\le n < \min\{n_*, \hat n\}$. Noting that $\xi_0^{\d_l} = \xi_0$, 
$x_0^{\d_l} = x_0$ and $\tilde \gamma_0^{\d_l} = \tilde \gamma_0 = 0$, and, as $m_0^{\d_l} = m_0 =0$, 
we have $\beta_0^{\d_l} = \beta_0 = 0$. Furthermore, noting that $r_0^{\d_l} = r_0$, we also have 
$\a_0^{\d_l} r_0^{\d_l} = \a_0 r_0$ no matter whether $r_0 = 0$ or not. Consequently (\ref{AHB.39}) 
holds for $n = 0$. Now we assume that (\ref{AHB.39}) holds for $0\le n <k$ for some integer 
$1\le k <\min\{n_*, \hat n\}$. By the induction hypothesis, the definition of $\xi_k^{\d_l}$, 
$x_k^{\d_l}$, and the continuity of $F$, $L$ and $\nabla \R^*$, we can easily derive that 
\begin{align*}
\xi_k^{\d_l} & = \xi_{k-1}^{\d_l} -\a_{k-1}^{\d_l} L(x_{k-1}^{\d_l})^* r_{k-1}^{\d_l} + \beta_{k-1}^{\d_l} m_{k-1}^{\d_l} \\
& \to \xi_{k-1} - \a_{k-1} L(x_{k-1})^* r_{k-1} + \beta_{k-1} m_{k-1} \\
& = \xi_k, \\
x_k^{\d_l} & = \nabla \R^*(\xi_k^{\d_l}) \to \nabla \R^*(\xi_k) = x_k
\end{align*}
and 
\begin{align*}
\tilde \gamma_k^{\d_l} & = \l m_k^{\d_l}, x_k^{\d_l} - x_{k-1}^{\d_l}\r - (1-\eta) \a_{k-1}^{\d_l} \|r_{k-1}^{\d_l}\|^2 
+ (1+\eta)\d_l \a_{k-1}^{\d_l} \|r_{k-1}^{\d_l}\| + \beta_{k-1}^{\d_l} \tilde \gamma_{k-1}^{\d_l} \\
& \to \l m_k, x_k - x_{k-1}\r - (1-\eta) \a_{k-1} \|r_{k-1}\|^2 
+ \beta_{k-1} \tilde \gamma_{k-1} = \tilde \gamma_k
\end{align*}
as $l \to \infty$. We now show $\a_k^{\d_l} r_k^{\d_l} \to \a_k r_k$ as $l \to \infty$. When $r_k = 0$, 
by using $0\le \a_k^{\d_l} \le \mu_1$ we have 
$$
\|\a_k^{\d_l} r_k^{\d_l} - \a_k r_k\| = \|\a_k^{\d_l} r_k^{\d_l}\| 
\le \mu_1 \|r_k^{\d_l} \| \to \mu_1 \|r_k\| =0 \quad 
\mbox{ as } l \to \infty.
$$
When $r_k \ne 0$, for $g_k := L(x_k)^*r_k$ we may use Assumption \ref{ass1} (d) to obtain  
\begin{align*}
\l g_k, x_k - x^\dag\r & = \l r_k, L(x_k)(x_k - x^\dag)\r \\
& = \|r_k\|^2 + \l r_k, y - F(x_k) - L(x_k)(x^\dag - x_k)\r \\
& \ge \|r_k\|^2 - \|r_k\| \|y - F(x_k) - L(x_k)(x^\dag - x_k)\|\\
& \ge (1-\eta) \|r_k\|^2 > 0
\end{align*}
which implies $g_k \ne 0$. Thus, by using $x_k^{\d_l} \to x_k$ and the continuity of $F$ 
and $L$, we have $\|r_k^{\d_l}\| > \tau \d_l$ and $g_k^{\d_l} \ne 0$ for sufficiently 
large $l$. Consequently, by the definition of $\a_k^{\d_l}$ and $\a_k$ we can conclude that 
$\a_k^{\d_l} \to \a_k$ as $l \to \infty$. Therefore 
$$
\|\a_k^{\d_l} r_k^{\d_l} - \a_k r_k\| \le \a_k^{\d_l} \|r_k^{\d_l} - r_k\| + |\a_k^{\d_l} - \a_k|\|r_k\|
\to 0 \quad \mbox{ as } l \to \infty.
$$
Recall that $m_k\ne 0$, we thus have $m_k^{\d_l} \ne 0$ for large $l$. Therefore, by using the definition of 
$\beta_k^{\d_l}$ and $\beta_k$, the above established facts, and  the induction hypothesis we can obtain 
\begin{align*}
\beta_k^{\d_l} 
& = \min\left\{\max\left\{0, \frac{\a_k^{\d_l} \l g_k^{\d_l}, m_k^{\d_l}\r - 2\sigma\tilde \gamma_k^{\d_l}}
{\|m_k^{\d_l}\|^2}\right\}, \beta\right\} \\
& \to \min\left\{\max\left\{0, \frac{\a_k \l g_k, m_k\r - 2\sigma\tilde \gamma_k}{\|m_k\|^2}\right\}, \beta\right\} 
=\beta_k
\end{align*}
as $l \to \infty$, where we used $\a_k^{\d_l} g_k^{\d_l} \to \a_k g_k$ as $l\to \infty$ which follows 
easily from the established facts. By the induction principle, we thus obtain (\ref{AHB.39}) for 
$0\le n < \min\{n_*, \hat n\}$. 
			
When $\hat n \ge n_*$, we next use an induction argument to show that 
\begin{align}\label{AHB.40}
\xi_n^{\d_l} \to \xi_n \quad \mbox{ and } \quad x_n^{\d_l} \to x_n \quad \mbox{ as } l \to \infty 
\end{align}
for $n_* \le n \le \hat n$. Recall $m_{n_*} =0$, we may use Lemma \ref{AHB.lem7} to conclude that $r_k =0$ 
and $m_{k+1} = 0$ for all $k \ge n_*-1$. We first show (\ref{AHB.40}) for $n = n_*$. Note that  
$$
\xi_{n_*} = \xi_{n_*-1} + \beta_{n_*-1} m_{n_*-1}.
$$
Therefore, by the definition of $\xi_{n_*}^{\d_l}$, (\ref{AHB.39}), $0\le \a_{n_*-1}^{\d_l} \le \mu_1$ 
we have 
\begin{align*}
\|\xi_{n_*}^{\d_l} - \xi_{n_*}\| 
& \le \|\xi_{n_*-1}^{\d_l} - \xi_{n_*-1}\| + \mu_1 \| L(x_{n_*-1}^{\d_l})^* r_{n_*-1}^{\d_l}\| \\
& \quad \, + \|\beta_{n_*-1}^{\d_l} m_{n_*-1}^{\d_l} - \beta_{n_*-1} m_{n_*-1}\| \\
& \to \mu_1 \|L(x_{n_*-1}) r_{n_*-1}\| = 0
\end{align*}
and consequently, by the continuity of $\nabla \R^*$, we have $x_{n_*}^{\d_l} \to x_{n_*}$ as 
$l \to \infty$. Now assume that (\ref{AHB.40}) holds for $n_*\le n \le k$ for some $n_*\le k < \hat n$. 
By using $r_k =0$ and $m_k=0$ we have $\xi_{k+1} = \xi_k$. Thus, by using $0\le \a_k^{\d_l} \le \mu_1$ 
and $0\le \beta_k^{\d_l} \le \beta$, we have 
\begin{align*}
\|\xi_{k+1}^{\d_l} - \xi_{k+1}\| 
& = \|\xi_k^{\d_l} - \xi_k - \a_k^{\d_l} L(x_k^{\d_l})^* r_k^{\d_l} + \beta_k^{\d_l} m_k^{\d_l}\| \\
& \le \|\xi_k^{\d_l} - \xi_k\| + \mu_1 \|L(x_k^{\d_l})^* r_k^{\d_l}\| + \beta \|m_k^{\d_l}\| \\
& \to \mu_1 \|L(x_k)^* r_k\| + \beta \|m_k\| = 0
\end{align*}
as $\d \to 0$. This shows that (\ref{AHB.40}) is also true for $n = k+1$. 

Putting the facts established in (\ref{AHB.39}) and (\ref{AHB.40}) together, the proof is 
therefore complete. 
\end{proof}

Finally we are ready to show the main strong convergence result on Algorithm \ref{alg:AHB} for 
solving (\ref{AHB.1}) using noisy data. 

\begin{theorem}\label{AHB.thm4}
Let Assumption \ref{ass0} and Assumption \ref{ass1} hold and consider Algorithm \ref{alg:AHB}, 
where $0\le \beta < 1$, and $\tau>1$ and $\mu_0 > 0$ are chosen such that (\ref{mutau}) holds. 
Let $n_\d$ be the output integer. Then there exists a solution $x^*$ of (\ref{AHB.1}) in 
$B_{2\rho}(x_0) \cap \emph{dom}(\R)$ such that 
$$
\lim_{\d\to 0} D_\R^{\xi_{n_\d}^\d} (x^*, x_{n_\d}^\d) = 0 \quad \mbox{ and } \quad 
\lim_{\d \to 0}\|x_{n_\d}^\d - x^*\|=0.
$$
If, in addition, $\emph{Ker}(L(x^\dag))\subset\emph{Ker}(L(x))$ for all $x\in B_{2\rho}(x_0)$, 
then $x^* = x^\dag$.
\end{theorem}
		
\begin{proof}
Let $x^*$ be the solution of (\ref{AHB.1}) in $B_{2\rho}(x_0) \cap \mbox{dom}(\R)$ determined 
in Theorem \ref{AHB.lem5} such that $D_\R^{\xi_n}(x^*, x_n) \to 0$ as $n\to \infty$, where 
$\{x_n\}$ denotes the sequence defined by the counterpart of Algorithm \ref{alg:AHB} using 
the exact data. Let $\Delta_n^\d: = D_\R^{\xi_n^\d} (x^*, x_n^\d)$ for $n \ge 0$. We will show 
that $\Delta_{n_\d}^\d \to 0$ as $\d \to 0$ by considering two cases. 
			
Assume first that there is a sequence $\{y^{\d_l}\}$ of noisy data satisfying $\|y^{\d_l}-y\|\le \d_l$ 
with $\d_l \to 0$ such that $n_l:= n_{\d_l}\to \hat{n}$ as $l\to \infty$ for some finite integer 
$\hat n$. Then $n_l = \hat{n}$ for all large $l$. According to the definition of $n_l:=n_{\d_l}$ 
we have
$$
\|F(x_{\hat{n}}^{\d_l}) - y^{\d_l}\| \le \tau \d_l.
$$
By taking $l\to\infty$ and using lemma \ref{AHB.lem6}, we can obtain $F(x_{\hat{n}}) = y$. Thus, 
we may use Lemma \ref{AHB.lem7} (i) to obtain $x_n = x_{\hat n}$ for all $n \ge \hat n$. Since 
$x_n \to x^*$ as $n \to \infty$, we must have $x_{\hat n} = x^*$ and thus, by Lemma \ref{AHB.lem6} 
and the lower semi-continuity of $\R$, we can obtain
\begin{align*}
\limsup_{l\to \infty} \Delta_{n_l}^{\d_l} 
& = \limsup_{l\to \infty} \Delta_{\hat n}^{\d_l} 
= \limsup_{l\to \infty} \left(\R(x_{\hat n}) - \R(x_{\hat n}^{\d_l}) 
- \l \xi_{\hat n}^{\d_l}, x_{\hat n} - x_{\hat n}^{\d_l}\r \right) \\
& = \R(x_{\hat n}) - \liminf_{l\to \infty} \R(x_{\hat n}^{\d_l}) 
- \lim_{l\to \infty} \l \xi_{\hat n}^{\d_l}, x_{\hat n} - x_{\hat n}^{\d_l}\r \\
& \le \R(x_{\hat n}) - \R(x_{\hat n}) = 0
\end{align*}
which shows that $\Delta_{n_l}^{\d_l} \to 0$ as $l \to \infty$. 
			
Assume next that there is a sequence $\{y^{\d_l}\}$ of noisy data satisfying $\|y^{\d_l}-y\|\le \d_l$ 
with $\d_l\to 0$ such that $n_l:= n_{\d_l}\to \infty$ as $l\to \infty$. Let $n$ be any fixed integer. 
Then $n_l>n$ for large $l$. It then follows from Lemma \ref{AHB.lem2} that 
$
\Delta_{n_l}^{\d_l} \le \Delta_n^{\d_l}.
$ 
By using Lemma \ref{AHB.lem6} and the lower semi-continuity of $\R$, we thus obtain
\begin{align*}
\limsup_{l\to \infty} \Delta_{n_l}^{\d_l} 
& \le \limsup_{l\to \infty}\Delta_n^{\d_l} 
= \limsup_{l\to \infty} \left(\R(\hat x) - \R(x_n^{\d_l}) - \l \xi_n^{\d_l}, x^* - x_n^{\d_l}\r \right) \\
& \le \R(x^*) - \R(x_n) - \l \xi_n, x^* - x_n\r = D_\R^{\xi_n}(x^*, x_n). 
\end{align*}
Letting $n\to \infty$ in the above equation gives $\limsup_{l\to \infty} \Delta_{n_l}^{\d_l} \le 0$ 
which shows again $\Delta_{n_l}^{\d_l} \to 0$ as $l \to \infty$.
\end{proof}

\section{\bf Numerical results}\label{sect4}

In this section we will provide numerical simulations to test the performance of our AHB method, 
i.e. Algorithm \ref{alg:AHB}. The computational results demonstrate that our AHB method 
indeed has superior performance over the Landweber iteration in terms of both the number 
of iterations and the CPU running time. Our simulations are performed via MATLAB R2023b 
on a Dell laptop with 11th Gen Intel(R) Core(TM) i7-1185G7 @ 3.00GHz 1.80 GHz and 32GB memory.

\begin{example}\label{ex1}
{\rm 
We first test the performance of Algorithm \ref{alg:AHB} by considering the following linear 
integral equation of the first kind
\begin{align}\label{IE.1}
(F x)(s) := \int_0^1 \kappa(s, t) x(t) dt = y(s), \quad s \in [0,1], 
\end{align}
where the kernel $\kappa(s,t)$ is a continuous function defined on $[0, 1] \times [0, 1]$. 
It is easy to see that $F$ is a compact linear map from $L^2[0, 1]$ to itself. Our goal is 
to determine the minimal-norm solution of (\ref{IE.1}) by using some  
noisy data $y^\d\in L^2[0,1]$ satisfying $\|y^\d - y\|_{L^2[0,1]} \le \d$ with a given noise 
level $\d>0$. When applying Algorithm \ref{alg:AHB} to find the minimal-norm solution, we use $x_0 = 0$, 
$\R(x) = \frac{1}{2}\|x\|_{L^2[0,1]}^2$, $\tau = 1.01$, $\beta = \infty$ and 
$$
\a_n \equiv \frac{\mu_0}{\|F\|^2} \quad \mbox{ with } \mu_0 = 0.99(2-2/\tau),
$$
where $\mu_0$ is chosen to fulfill the theoretical requirement. As comparisons, we also consider
the following algorithms:

\begin{enumerate}[leftmargin = 0.8cm]
\item[$\bullet$] Landweber iteration 
$$
x_{n+1}^\d = x_n^\d - \a F^*(F x_n^\d - y^\d)
$$
with the initial guess $x_0^\d = 0$ terminated by the discrepancy principle 
\begin{align}\label{DP1}
\|F x_{n_\d}^\d - y^\d\| \le \tau \d < \|F x_n^\d - y^\d\|, \quad 0\le n < n_\d.
\end{align}
This method is convergent for $0< \a \le 2/\|A\|^2$ and $\tau > 1$; see \cite{EHN1996}. 
In our computation we use $\a = 1/\|A\|^2$ and $\tau = 1.01$. 

\item[$\bullet$] The $\nu$-method of Brakhage: This method was introduced in \cite{B1987}
to accelerate the Landweber iteration. It preassigns a number $\nu > 1/2$ and defines the 
iterates by  
$$
x_{n+1}^\d = x_n^\d - \a_n \gamma F^*(F x_n^\d - y^\d) + \beta_n (x_n^\d - x_{n-1}^\d)
$$
with $x_{-1}^\d = x_0^\d =0$ and 
\begin{align*}
\quad \a_n = 4 \frac{(2n+2\nu+1)(n+\nu)}{(n+2\nu)(2n+4\nu+1)}, \, \, \,
\beta_n = \frac{n(2n-1)(2n+2\nu+1)}{(n+2\nu)(2n+4\nu+1)(2n+2\nu-1)}.
\end{align*}
It is known that it is a regularization method for $0< \gamma < 1/\|F\|^2$ when terminated by 
the discrepancy principle (\ref{DP1}) with $\tau >1$; see \cite{B1987,EHN1996,H1991}. We 
use $\nu =3$, $\gamma = 0.99/\|F\|^2$ and $\tau = 1.01$.  

\item[$\bullet$] Nesterov acceleration: This method was proposed in \cite{N1983} to accelerate the 
gradient method for convex optimization problems. This strategy was then suggested in \cite{Jin2016} 
to accelerate the Landweber iteration for ill-posed inverse problems. The corresponding method for 
solving (\ref{IE.1}) takes the form
\begin{align*}
z_n^\d = x_n^\d + \frac{n-1}{n+\a} (x_n^\d - x_{n-1}^\d), \quad 
x_{n+1}^\d = z_n^\d - \gamma F^*(F z_n^\d - y^\d)
\end{align*}
with $x_{-1}^\d = x_0^\d =0$ and $\a\ge 2$. When the method is terminated by the discrepancy 
principle (\ref{DP1}) with $\tau >1$, the corresponding regularization property has been proved 
in \cite{K2021,N2017} for $0< \gamma < 1/\|F\|^2$. We use $\a = 3$, $\gamma = 0.99/\|F\|^2$ and 
$\tau = 1.01$.
\end{enumerate}

\begin{table}[ht]
\caption{Numerical results for Example \ref{ex1}.} \label{table1}
    \begin {center}
\begin{tabular}{lllll}
     \hline
$\d$  \qquad \quad \quad   & method \qquad \qquad   & iterations \quad \quad & time (s) \quad \quad & relative error \\ \hline
0.1     & Landweber                 & 62    & 0.0153 & 1.9024e-2 \\
        & $\nu$-method              & 15    & 0.0083 & 1.6657e-2 \\
        & Nesterov                  & 16    & 0.0091 & 1.7684e-2 \\
        & AHB                       & 20    & 0.0087 & 1.7774e-2 \\ 
0.01    & Landweber                 & 190   & 0.0359 & 5.0988e-3 \\
        & $\nu$-method              & 30    & 0.0102 & 5.1081e-3 \\
        & Nesterov                  & 41    & 0.0157 & 4.1816e-3 \\
        & AHB                       & 30    & 0.0108 & 4.6982e-3 \\ 
0.001   & Landweber                 & 1256  & 0.1934 & 1.8305e-3 \\
        & $\nu$-method              & 82    & 0.0189 & 1.8587e-3 \\
        & Nesterov                  & 102   & 0.0285 & 1.7776e-3 \\
        & AHB                       & 80    & 0.0187 & 1.7220e-3 \\ 
0.0001  & Landweber                 & 8144  & 1.2922 & 6.2937e-4 \\
        & $\nu$-method               & 215   & 0.0414 & 6.4159e-4 \\
        & Nesterov                  & 324   & 0.0831 & 5.1961e-4 \\
        & AHB                       & 268   & 0.0513 & 6.1130e-4 \\ 
0.00001 & Landweber                 & 55145 & 8.2924 & 1.9023e-4 \\
        & $\nu$-method              & 565   & 0.0918 & 1.9305e-4 \\
        & Nesterov                  & 817   & 0.2003 & 1.7054e-4 \\
        & AHB                       & 896   & 0.1520 & 1.8902e-4 \\ \hline 
\end{tabular}\\[5mm]
\end{center}
\end{table}

In our numerical computation, we consider the equation (\ref{IE.1}) with the kernel 
\begin{align*}
\kappa(s,t) = \left\{\begin{array}{lll}
40 s(1-t) & \mbox{ if } s\le t, \\
40t(1-s) & \mbox{ if } t \le s
\end{array} \right.
\end{align*}
and assume that the sought solution is $x^\dag(t) = 4t(1-t) + \sin(2\pi t)$. Let 
$y := A x^\dag$ be the exact data. We add random noise to $y$ to produce a noisy
data $y^\d$ satisfying $\|y^\d - y\|_{L^2[0,1]} = \d$ for various noise levels 
$\d>0$. In our implementation, all integrals over $[0,1]$ are approximated by 
the trapezoidal rule based on the $1000$ nodes partitioning $[0,1]$ into 
subintervals of equal length.

For this model problem, we execute the Landweber iteration, the $\nu$-method, the Nesterov acceleration 
and our AHB method with the above setup using noisy data with various noise levels and the computational 
results are reported in Table \ref{table1} which includes the required number of iterations and the 
consumed CPU time together with the corresponding relative errors. These results indicate that 
all the four methods can produce comparable approximate solutions in terms of accuracy. However, 
the $\nu$-method, the Nesterov acceleration and our AHB method clearly demonstrate superior performance 
over the Landweber method by significantly reducing the number of iterations and the CPU running time. 
In order to visualize how the iterates approach the sought solution, in Figure \ref{ex1_relerr} we 
plot the relative error $\|x_n^\d - x^\dag\|_{L^2}/\|x^\dag\|_{L^2}$ versus $n$ for these four methods: 
the left figure plots the relative error, using noisy data with $\d = 0.001$, until 
the iteration is terminated by the discrepancy principle; the right figure plots the relative error 
when all computations are performed using the exact data. These plots further indicate the acceleration 
effect of the $\nu$-methd, the Nesterov acceleration and our AHB method. Moreover, when using the 
exact data, our AHB method can outperform both the $\nu$-method and the Nesterov acceleration. 

\begin{figure}[ht]
\centering
\includegraphics[width = 0.48\textwidth]{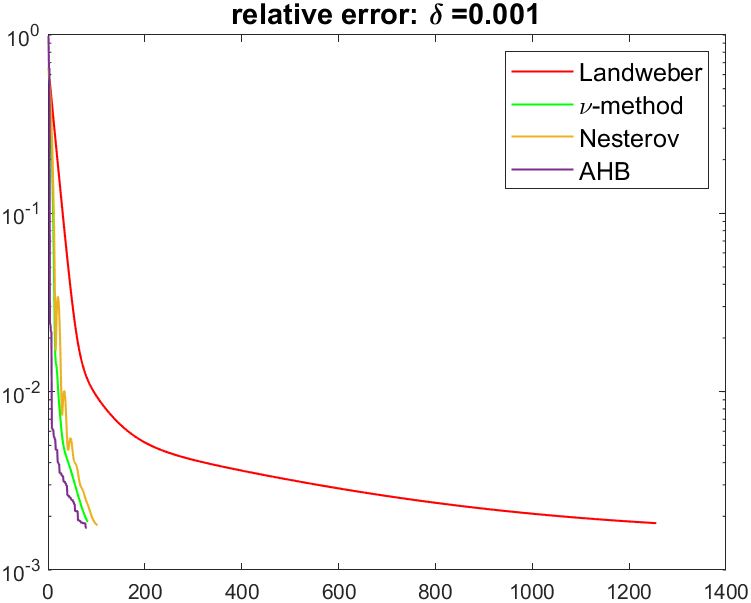}
\includegraphics[width = 0.48\textwidth]{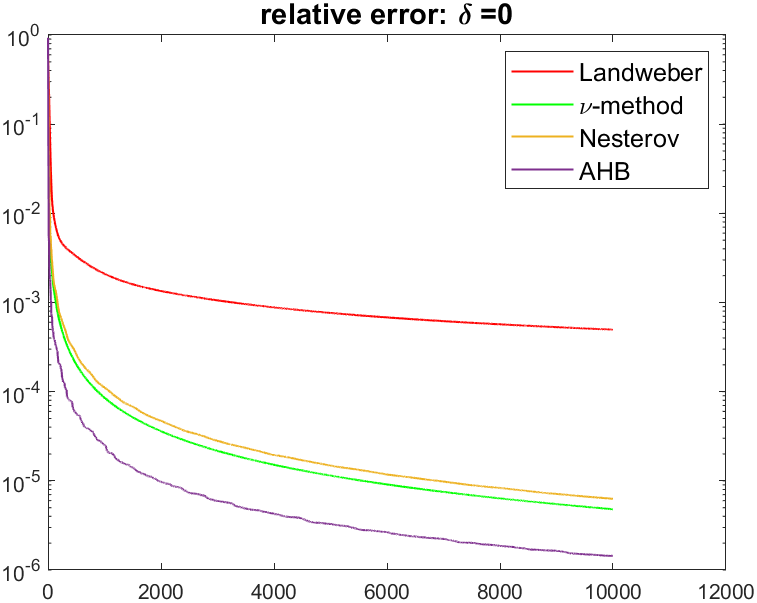}
\caption{Relative error versus the number of iterations for Example \ref{ex1}. }\label{ex1_relerr}
			\end{figure}

When using noisy data in computation, among the three fast methods, the $\nu$-method usually 
requires least number of iterations. However, it is totally unclear how to extend this method 
to solve (\ref{Rmin}) with nonlinear $F$ or non-quadratic $\R$. The Nesterov acceleration 
in general takes longer time than the other two methods due to the fact that, at each iteration 
step, it needs to calculate $A z_n^\d - y^\d$ to update $x_{n+1}^\d$ and $A x_n^\d - y^\d$ to 
check the discrepancy principle as required by the theory (\cite{K2021}) and thus the computational 
complexity per iteration increases. Although we have formulated a possible extension (\ref{NAG}) 
of Nesterov acceleration to solve (\ref{Rmin}) with general $F$ and $\R$ and numerical simulations 
demonstrate the striking performance, a solid theoretical justification however is still 
unavailable. Note that, in our AHB method, we used a very small step size, compared with the ones 
used by other three methods. With an adaptive choice of the momentum coefficient, it still 
achieves an excellent acceleration effect on Landweber iteration. Moreover, our theory provides 
the convergence guarantee on using our AHB method to solve (\ref{Rmin}) with general $F$ and $\R$. 
In the following numerical examples, we will focus on testing the acceleration effect of our AHB 
method over the Landweber iteration. 

}
\end{example}

\begin{example}[Computed tomography]
{\rm 
In this example, we consider applying Algorithm \ref{alg:AHB} to computed tomography which consists 
in determining the density of cross sections of a human body by measuring the attenuation of X-rays 
as they propagate through the biological tissues (\cite{N2001}).

Assume the image is supported on a rectangular domain in $\mathbb{R}^2$
and is discretized on a $I \times J$ pixel grid. We will identify any image 
of size $I\times J$ by a long vector in ${\mathbb R}^N$ with $N = I\times J$ 
by stacking all its columns. We consider the standard 2D fan-beam tomography
using $n_\theta$ projection angles evenly distributed between $1$ and $360$ degrees 
with $p$ lines of X-rays emitted through each angle. We use the function ``\texttt{fanbeamtomo}” 
in the MATLAB package AIR TOOLS \cite{HS2012} to discretize the problem, leading 
to an ill-conditioned linear algebraic system $F x = y$, where $F \in \mathbb{R}^{M\times N}$ 
with $M = n_\theta p$ and $N = I\times J$.

\begin{table}[ht]
\caption{Numerical results for computed tomography.} \label{table2}
    \begin {center}
\begin{tabular}{lllll}
     \hline
$\d_{rel} \quad $ & method \qquad \qquad \quad    & iterations \qquad & time (s) \quad \quad & relative error \\ \hline
0.05     & Landweber             & 101    & 9.9940  & 1.8538e-01 \\         
         & AHB: $\beta = 0.99$   & 63     & 5.5243  & 1.9299e-01 \\
         & AHB: $\beta = \infty$ & 48     & 4.2141  & 1.9060e-01 \\ 
0.01     & Landweber             & 444    & 44.0053 & 5.9528e-02 \\
         & AHB: $\beta = 0.99$   & 230    & 19.7274 & 5.9464e-02 \\
         & AHB: $\beta = \infty$ & 238    & 20.2110 & 5.9760e-02 \\ 
0.005    & Landweber             & 778    & 78.4099 & 3.1650e-02 \\
         & AHB: $\beta = 0.99$   & 335    & 27.8164 & 3.1561e-02 \\
         & AHB: $\beta = \infty$ & 374    & 31.1740 & 3.1859e-02 \\ 
0.001    & Landweber             & 3254   & 366.628 & 6.2999e-03 \\
         & AHB: $\beta = 0.99$   & 1398   & 124.763 & 5.7077e-03 \\
         & AHB: $\beta = \infty$ & 1495   & 132.288 & 5.9425e-03 \\ \hline 
\end{tabular}\\[5mm]
\end{center}
\end{table}

In our numerical simulations, we consider the modified Shepp-Logan phantom of size $256\times 256$ 
generated by MATLAB, we also use $n_\theta = 60$ projection angles with $p = 367$ lines of X-rays per 
projection. Correspondingly $F$ has the size $M\times N$ with $M = 22020$ and $N = 65536$. Instead of 
the exact data $y$, we add Gaussian noise on $y$ to generate a noisy data $y^\d$ with relative noise 
level $\d_{rel} = \|y^\d - y\|_2/\|y\|_2$ so that the noise level is $\d = \d_{rel} \|y\|_2$. In order 
to capture the feature of the sought image, we take 
$$
\R(x) = \frac{1}{2\kappa} \|x\|_F^2 + |x|_{TV}, \quad \forall x\in {\mathbb R}^{I\times J}
$$ 
with $\kappa = 1$, where $\|x\|_F$ is the Frobenius norm and $|x|_{TV}$ denotes the isotropic total variation of $x$, i.e. $|x|_{TV} = f(\nabla x)$ with $\nabla$ being the discrete gradient operator
defined by $\nabla x = (\nabla_1 x,\nabla_2x)$ for $x = (x_{i,j}) \in {\mathbb R}^{I\times J}$, 
where 
\begin{align*}
(\nabla_1 x)_{i,j}&=
\left\{
  \begin{array}{ll}
   x_{i+1,j}-x_{i,j}, &  i = 1,\cdots,I-1; \, j = 1,\cdots,J\\
    x_{1,j}-x_{I,j}, &  i = I; \, j = 1,\cdots,J ;\\
  \end{array}
\right.\\
(\nabla_2 x)_{i,j}&=
\left\{
  \begin{array}{ll}
    x_{i,j+1}-x_{i,j}, & i = 1,\cdots,I; \, j = 1,\cdots,J-1,\\
    x_{i,1}-x_{i,J}, & i =1,\cdots, I-1; \, j = J
  \end{array}
\right.
\end{align*}
and
\begin{align*}
f(u,v) 
= \sum_{i=1}^{I}\sum_{j=1}^J \sqrt{u_{i,j}^2+v_{i,j}^2}, \quad 
\forall u =(u_{i,j}), v=(v_{i,j}) \in {\mathbb R}^{I\times J}.
\end{align*}
Clearly $\R$ satisfies Assumption \ref{ass0} with $\sigma = 1/(2\kappa)$. When using 
Algorithm \ref{alg:AHB} to reconstruct the image, we use the initial guess $\xi_0 = 0$ and the step-size 
$$
\a_n^\d = \min\left\{\frac{\mu_0\|F x_n^\d - y^\d\|^2}{\|F^*(F x_n^\d - y^\d)\|^2}, \mu_1\right\}
$$
Our convergence theory requires $\tau > 1$, $0 <\mu_0 < (2 - 2/\tau)/\kappa$ and $\mu_1 >0$. Therefore 
we take $\tau = 1.05$, $\mu_0 = 0.99 (2-2/\tau)/\kappa$ and $\mu_1 = 100$. In order to determine the 
momentum coefficient $\beta_n^\d$, we consider two values of $\beta$: $\beta = 0.99$ and $\beta = \infty$. 
As comparison, we also carry out the computation by the Landweber method (\ref{Land}) using the same step 
size $\a_n^\d$ and the regularization function $\R$. During the computation, updating $x_n^\d$ from 
$\xi_n^\d$ requires to solving the total variation denoising problem
$$
x_n^\d = \arg\min_{x\in {\mathbb R}^{I\times J}} 
\left\{\frac{1}{2\kappa} \|x - \kappa \xi_n^\d\|_F^2 + |x|_{TV} \right\}
$$
which is solved approximately by the primal-dual hybrid gradient (PDHG) method (\cite{BR2012,Jin2016,ZC2008}) 
after $70$ iterations.

\begin{figure}[ht]
\centering
\includegraphics[width = 0.24\textwidth]{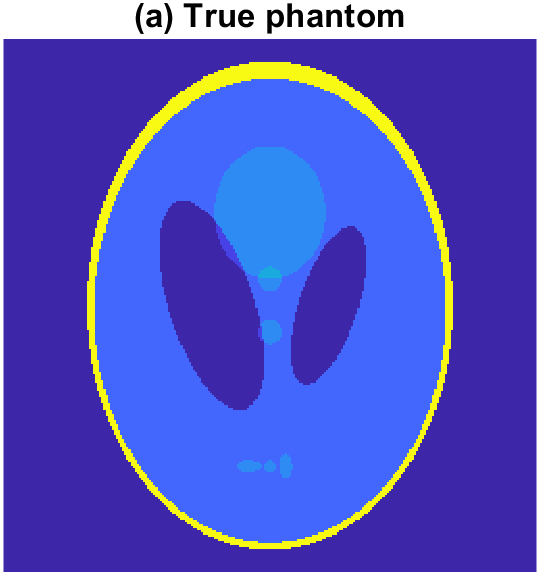}
\includegraphics[width = 0.24\textwidth]{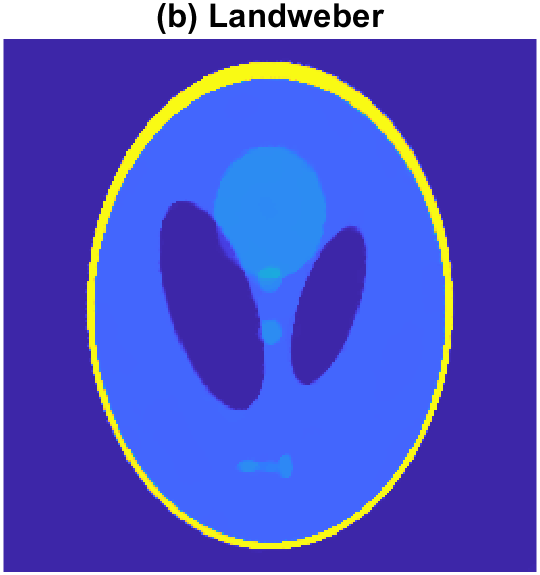}
\includegraphics[width = 0.24\textwidth]{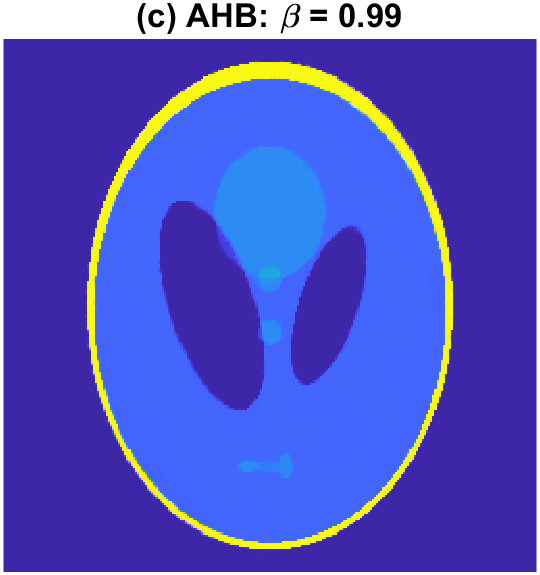}
\includegraphics[width = 0.24\textwidth]{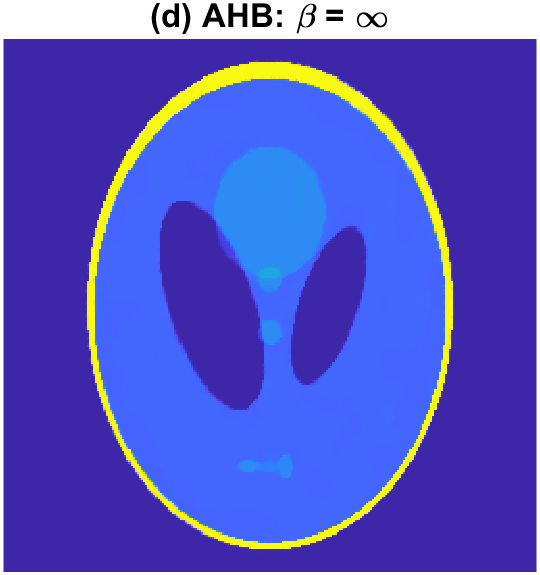}
\caption{The reconstruction results for computed tomography. (a) True phantom; (b) Landweber; 
(c) AHB with $\beta = 0.99$; (d) AHB with $\beta = \infty$.}\label{ex2_reconstruction}
\end{figure}

The computational results by AHB and Landweber method are reported in Table \ref{table2}, including the 
number of iterations $n_\d$, the CPU running time and the relative errors $\|x_{n_\d}^\d-
x^\dag\|_2/\|x^\dag\|_2$, using noisy data with various relative noise level $\d_{rel}>0$. Table 
\ref{table2} shows that AHB with both $\beta = 0.99$ and $\beta = \infty$ leads to a considerable decrease 
in the number of iterations and the amount of computational time, which demonstrates the striking 
acceleration effect of our AHB method. 

\begin{figure}[ht]
\centering
\includegraphics[width = 0.7\textwidth]{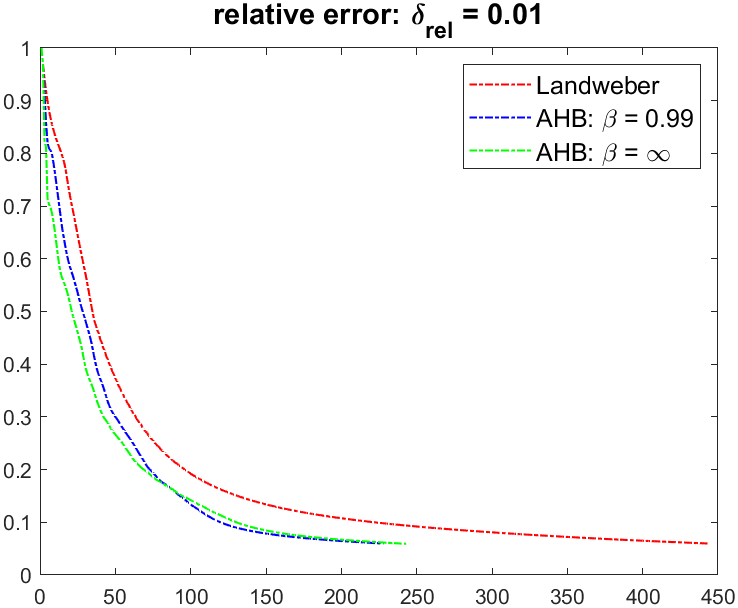}
\caption{Computed tomography: relative errors versus the number of iterations.}\label{ex2_relerr}
\end{figure}

In order to visualize the reconstruction accuracy of our AHB method, we plot in Figure 
\ref{ex2_reconstruction} the true image, the reconstruction result by Landweber method and the AHB methods 
with $\beta = 0.99$ and $\infty$ using noisy data with relative noise level $\d_{rel} = 0.01$. In 
Figure \ref{ex2_relerr} we also plot the curves of the relative errors $\|x_n^\d-x^\dag\|_2/\|x^\dag\|_2$ 
versus $n$ for Landweber and our AHB methods. These plots further demonstrate that, compared with 
Landweber iteration, our AHB method can produce comparable reconstruction results with significantly 
less number of iterations. 
}
\end{example}

\begin{example}[Elliptic parameter identification]
 
{\rm In this example we consider the identification of the parameter $c$ in the elliptic boundary value problem
\begin{align}\label{PDE}
-\triangle u + cu = f \hbox{ in } \Omega, \qquad 
u=g  \hbox{ on } \partial \Omega 
\end{align}
from an $L^2(\Omega)$-measurement of the state $u$, where $\Omega\subset\mathbb{R}^d$ with $d\leq 3$ is
a bounded domain with Lipschitz boundary $\p \Omega$, $f\in H^{-1}(\Omega)$ and $g\in H^{1/2}(\Omega)$. 
We assume that the sought parameter $c^{\dag}$ is in $L^2(\Omega)$. This problem reduces to solving 
$F(c) =u$ if we define the nonlinear operator $F: L^2(\Omega)\rightarrow  L^2(\Omega)$ by $F(c): = u(c)$, 
where $u(c)\in H^1(\Omega)\subset  L^2(\Omega)$ is the unique solution of (\ref{PDE}). This operator $F$
is well defined on
\begin{equation*}
{\mathcal D}: = \left\{ c\in  L^2(\Omega) : \|c-\hat{c}\|_{ L^2(\Omega)}\leq \ep_0 \mbox{ for some }
\hat{c}\geq0,\ \textrm{a.e.}\right\}
\end{equation*}
for some positive constant $\ep_0>0$. It is known that the operator $F$ is weakly closed
and Fr{\'{e}}chet differentiable, the Fr\'{e}chet derivative of $F$ and its adjoint are given by 
$$
F'(c) h = - A(c)^{-1} (h u(c))    \quad \mbox{ and } \quad F'(c)^* w = - u(c) A(c)^{-1} w
$$
for $c\in {\mathcal D}$ and $h, w \in L^2(\Omega)$, where $A(c): H_0^1(\Omega) \to H^{-1}(\Omega)$ 
is defined by $A(c) u = - \triangle u  + c u$ which is an isomorphism uniformly in the ball 
$B_{2\rho}(c^\dag)$ for small $\rho >0$. Moreover, $F$ satisfies Assumption \ref{ass1} (d) with 
a number $\eta>0$ which can be very small if $\rho>0$ is sufficiently small (see \cite{EHN1996}).

\begin{table}[ht]
\caption{Numerical results for elliptic parameter identification.} \label{table3}
    \begin {center}
\begin{tabular}{lllll}
     \hline
$\d$ & method                & iterations \quad & time (s) \qquad & $\|x_{n_\d}^\d - x^\dag\|_{L^2}$ \\ \hline
0.005      & Landweber             & 36         & 6.2252   & 3.0388e-01 \\         
           & AHB: $\beta = 0.99$   & 15         & 2.6239   & 3.0577e-01 \\
           & AHB: $\beta = \infty$ & 12         & 2.1608   & 3.0328e-01 \\ 
0.001      & Landweber             & 248        & 39.967   & 1.3998e-01 \\
           & AHB: $\beta = 0.99$   & 64         & 10.509   & 1.3772e-01 \\
           & AHB: $\beta = \infty$ & 79         & 12.530   & 1.3790e-01 \\ 
0.0005     & Landweber             & 410        & 66.806   & 1.1268e-01 \\
           & AHB: $\beta = 0.99$   & 108        & 17.023   & 1.1298e-01 \\
           & AHB: $\beta = \infty$ & 165        & 26.356   & 1.1221e-01 \\ 
0.0001     & Landweber             & 2014       & 341.75   & 8.1312e-02 \\
           & AHB: $\beta = 0.99$   & 236        & 38.464   & 8.1573e-02 \\
           & AHB: $\beta = \infty$ & 375        & 60.491   & 8.1716e-02 \\ 
0.00005    & Landweber             & 4999       & 987.75   & 7.2499e-02  \\
           & AHB: $\beta = 0.99$   & 472        & 74.040   & 7.1581e-02  \\ 
           & AHB: $\beta = \infty$ & 821        & 133.82   & 7.2092e-02  \\ \hline 
\end{tabular}\\[5mm]
\end{center}
\end{table}

In our numerical simulation, we consider the two-dimensional problem with $\Omega = [0, 1]\times [0, 1]$ 
and the sought parameter $c^\dag$ is assumed to be a piecewise constant function as shown in 
Figure \ref{ex3_reconstruction} (a). Assuming $u(c^\dag) = x + y$, we add random noise to 
produce noisy data $u^\d$ satisfying $\|u^\d - u(c^\dag)\|_{L^2(\Omega)} = \d$ with various 
noise level $\d > 0$. We will use $u^\d$ to reconstruct $c^\dag$. In order to capture the feature 
of the sought parameter, we take 
\begin{align}\label{TV}
\R(c) = \frac{1}{2\kappa} \|c\|_{L^2(\Omega)}^2 + |c|_{TV}
 \end{align}
with the constant $\kappa = 10$, where $|c|_{TV}$ denotes the total variation of $c$ on $\Omega$, i.e.
$$
|c|_{TV} := \sup\left\{\int_\Omega c \, \mbox{div} \varphi \, \mbox{d}x \mbox{d}y: 
\varphi\in C_0^1(\Omega, {\mathbb R}^2) 
\mbox{ and } \|\varphi\|_{L^\infty(\Omega)} \le 1\right\}.
$$
It is easy to see that this $\R$ satisfies Assumption \ref{ass0} with $\sigma = 1/(2\kappa)$.

\begin{figure}[ht]
\centering
\includegraphics[width = 0.24\textwidth]{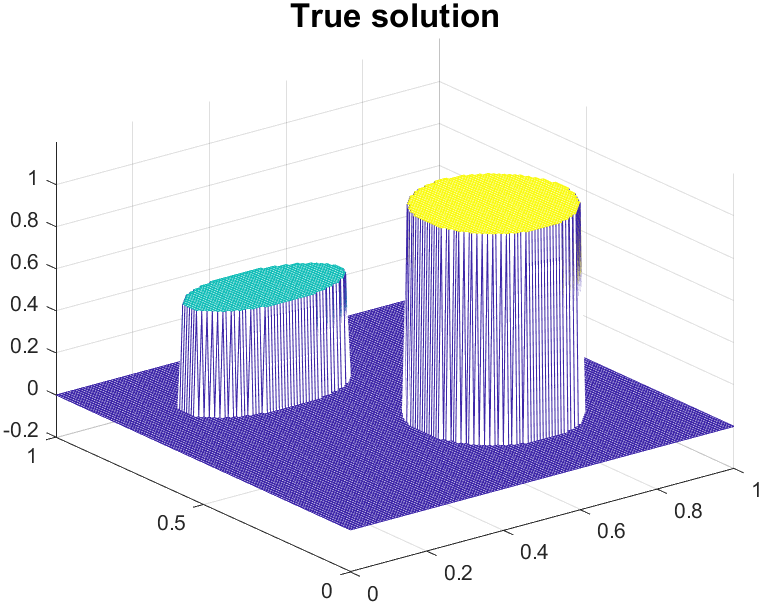}
\includegraphics[width = 0.24\textwidth]{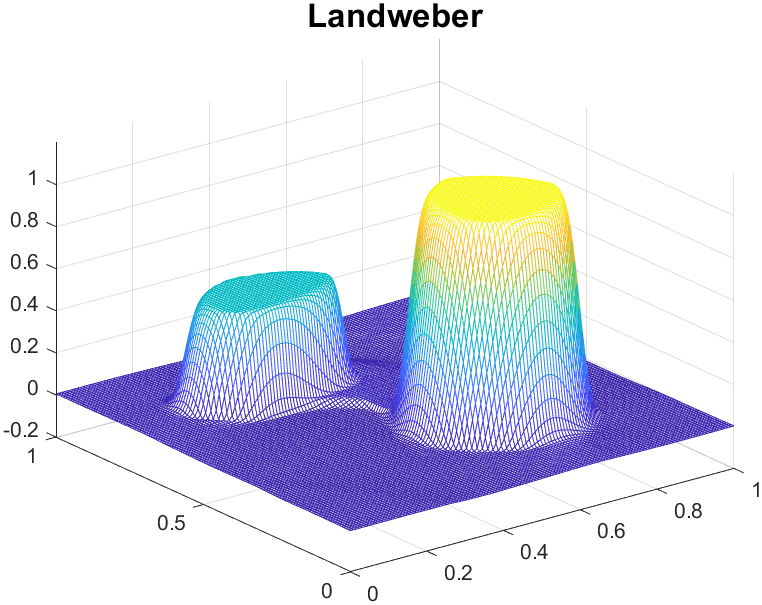}
\includegraphics[width = 0.24\textwidth]{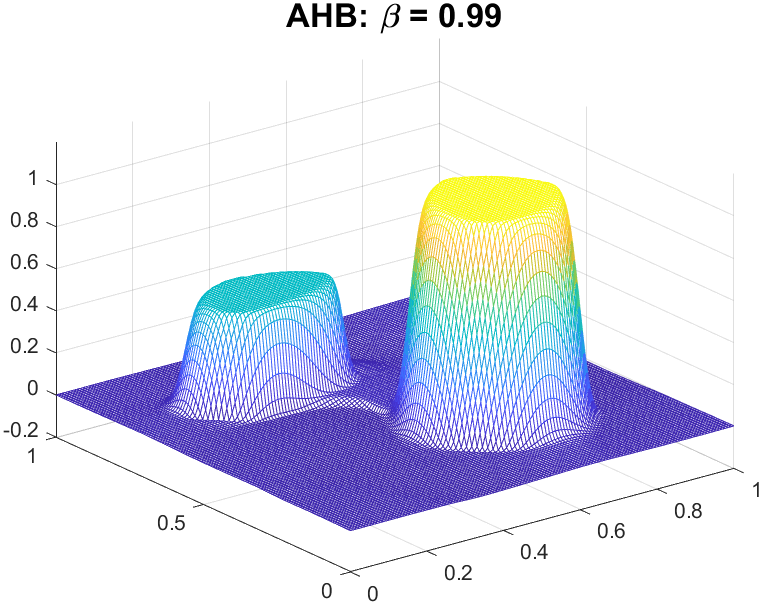}
\includegraphics[width = 0.24\textwidth]{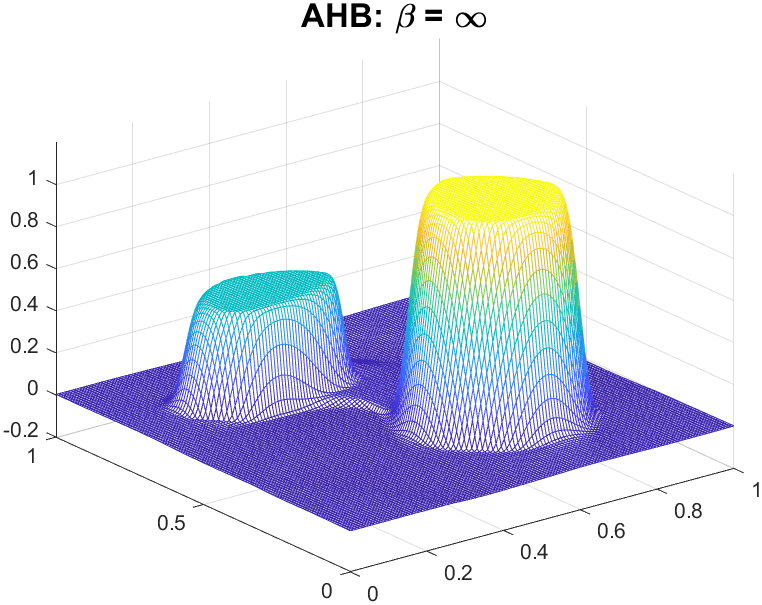}
\caption{The reconstruction results for parameter identification in elliptic equation. 
(a) True parameter; (b) Landweber; (c) AHB with $\beta = 0.99$; (d) AHB with $\beta = \infty$.}
\label{ex3_reconstruction}
\end{figure}

We first carry out the computation by the Landweber-type method (\ref{Land}) using the 
initial guess $\xi_0^\d = 0$ and adaptive step-size $\a_n^\d$ given by (\ref{ass}) with 
$\mu_0 = 1.96(1-\eta - (1+\eta)/\tau)/\kappa$ and $\mu_1 = 80$, where we take $\eta = 0.01$ 
and $\tau = 1.05$. Next we execute our AHB method, i.e. Algorithm \ref{alg:AHB}, with the 
same initial guess $\xi_0^\d$ and step size $\a_n^\d$, the momentum coefficient
$\beta_n^\d$ is determined by using two vules of $\beta$: $\beta = 0.99$ and $\beta = \infty$. 
In order to carry out the computation, we divide $\Omega$ into $128\times 128$ small squares
of equal size and solve all partial differential equations involved approximately by a multigrid 
method (\cite{H2016}) via finite difference discretization. Furthermore, updating $c_n^\d$ 
from $\xi_n^\d$ at each iteration step is equivalent to a total variation denoising problem 
which is solved by the PDHG method after $200$ iterations.

\begin{figure}[ht]
\centering
\includegraphics[width = 0.7\textwidth]{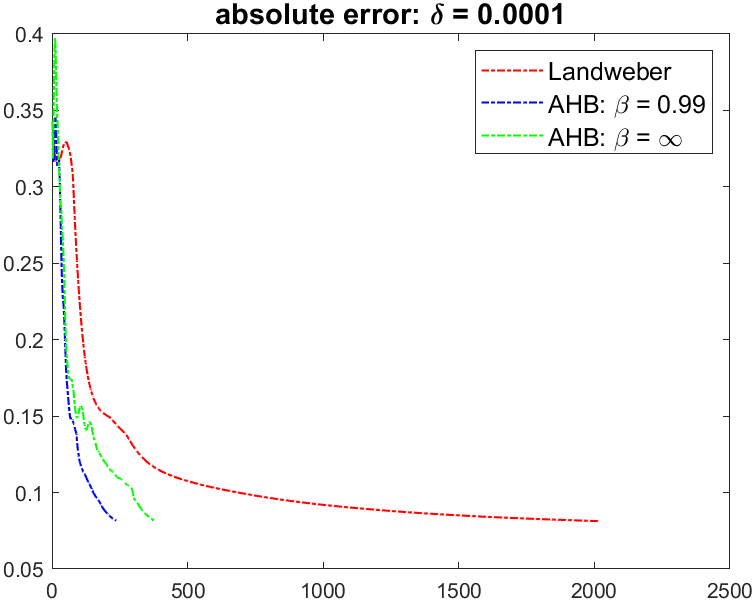}
\caption{relative errors versus the number of iterations.}\label{ex3_err}
			\end{figure}

In Table \ref{table3} we report the computational results by the Landweber iteration (\ref{Land}) 
and our AHB method with $\beta = 0.99$ and $\beta = \infty$, including the required number of 
iterations $n_\d$, the CPU running time and the absolute errors$\|c_{n_\d}^\d - c^\dag\|_{L^2(\Omega)}$, 
using noisy data for various noise level $\d>0$, which clearly demonstrates the acceleration 
effect of our AHB method and shows that our AHB method has superior performance over the Landweber 
iteration. In Figure \ref{ex3_reconstruction} and Figure \ref{ex3_err} we plot the respective 
reconstructed results and absolute errors versus the number of iteration by the Landweber iteration 
and our AHB method using noisy data with noise level $\d = 0.0001$, which shows these methods can 
produce comparable approximate solutions and indicates that our AHB method is much faster than 
the Landweber iteration.
}
\end{example}

\section{\bf Conclusion}

With the surge in machine learning advancements and the emergence of large-scale problems across 
diverse domains, Polyak's heavy ball method has experienced a resurgence of interest within the 
optimization community in recent years. Much effort has been dedicated to comprehending its 
acceleration effects. In this paper we 
proposed an adaptive heavy ball method for solving ill-posed inverse problems, both linear and 
nonlinear. This method differs from the Landweber-type method in that a momentum term is added to 
define the iterates. Our method integrates a strongly convex function into the design of the algorithm 
as a regularization function to detect the desired feature of the sought solution. 
Moreover, we introduced novel techniques for adaptively selecting step-sizes and momentum coefficients, 
aiming to achieve potential acceleration over the Landweber-type method. Notably, the updating formulae 
for these parameters are explicit, obviating the need for a backtracking line search procedure and 
thereby saving computational time. Under the discrepancy principle, we showed that our 
method can be terminated after a finite number of iterations and established the corresponding 
regularization property. We reported diverse numerical results which indicate the significant reduction 
in both the required number of iterations and computational time, and thus demonstrate the superior 
performance of our method over the Landweber-type method.

\section*{\bf Acknowledgement}

\noindent
The work of Q. Jin is partially supported by the Future Fellowship of the Australian Research Council (FT170100231).
The work of Q. Huang is supported by the China Scholarship Council program.

\bibliographystyle{amsplain}

\end{document}